\documentclass[11pt]{article}
\usepackage{hyperref}
\usepackage{bbm}
\usepackage{amssymb,amsmath,amsthm, fullpage}
\usepackage{color}
\usepackage{tikz}
\usepackage[mathscr]{euscript}
\usepackage{enumerate}
\usepackage[conditional,light,first,bottomafter]{draftcopy}
\draftcopyName{DRAFT\space\today}{130}
\draftcopySetScale{65}
\usepackage[letterpaper,hmargin=3.3cm,vmargin={2.6cm,3.2cm}]{geometry}
\usepackage{setspace}
\makeatletter
\renewcommand{\section}{\@startsection%
{section}%
{1}%
{0em}%
{1.7em}%
{1.2em}%
{\normalfont\large\centering\bfseries}}
\renewcommand{\@seccntformat}[1]%
{\csname the#1\endcsname.\hspace{0.5em}}
\makeatother
\renewcommand{\thesection}{\arabic{section}}

\numberwithin{equation}{section}
\renewcommand\appendix{\par
\setcounter{section}{0}%
\setcounter{subsection}{0}%
\setcounter{theorem}{0}
\setcounter{table}{0}
\setcounter{figure}{0}
\gdef\thetable{\Alph{table}}
\gdef\thefigure{\Alph{figure}}
\section*{Appendix}
\gdef\thesection{\Alph{section}}
\setcounter{section}{1}}
\newtheorem{theorem}{Theorem}[section]
\newtheorem{proposition}[theorem]{Proposition}
\newtheorem{lemma}[theorem]{Lemma}

\theoremstyle{definition}
\newtheorem{definition}{Definition}
\newtheorem{remark}{Remark}

\newcommand{\reals}{\mathbb{R}}
\newcommand{\nats}{\mathbb{N}}

\newcommand{\complex}{\mathbb{C}}

\newcommand{\norm}[1]{\left\|#1\right\|}
\newcommand{\I}{{\rm i}}
\newcommand{\inner}[2]{\left\langle#1,#2\right\rangle}

\newcommand{\cH}{{\cal H}}
\newcommand{\cc}[1]{\overline{#1}}
\newcommand{\eval}[1]{\upharpoonright_{#1}}
\newcommand{\convergesto}[2]{\xrightarrow[#1\to #2]{}}

\DeclareMathOperator{\im}{Im}
\DeclareMathOperator{\dom}{dom}

\DeclareMathOperator{\diag}{diag}

\DeclareMathOperator{\ran}{ran}

\DeclareMathOperator{\clos}{clos}
\DeclareMathOperator*{\slim}{s-lim}
\begin{document}
\title{\sc Functional model for extensions of symmetric operators and
  applications to scattering theory
\footnotetext{%
Mathematics Subject Classification (2010):
47A45 
34L25 
81Q35 
}
\footnotetext{%
Keywords:
Functional model;
Extensions of symmetric operators; Boundary triples;
Inverse scattering problems
}
}
\author{
\textbf{Kirill D. Cherednichenko}
\\
\small Department of Mathematical Sciences\\[-1.1mm]
\small University of Bath\\[-1.1mm]
\small 
Claverton Down, Bath BA2 7AY, U.K.\\[-1.1mm]
\small 
\texttt{K.Cherednichenko@bath.ac.uk}
\\[2mm]
\textbf{Alexander V. Kiselev}
\\
\small 
Institute of Physics and Mathematics,
\\[-1.1mm]
\small 
Dragomanov National Pedagogical University,
\\[-1.1mm]
\small 
9 Pyrohova St, Kyiv, 01601, Ukraine,
\\[-1.1mm]
\small 
\texttt{alexander.v.kiselev@gmail.com}
\\[2mm]
\textbf{Luis O. Silva}
\\
\small 
Departamento de F\'{i}sica Matem\'{a}tica\\[-1.2mm]
\small 
Instituto de Investigaciones en Matem\'aticas Aplicadas y en Sistemas\\[-1.1mm]
\small 
Universidad Nacional Aut\'onoma de M\'exico\\[-1.1mm]
\small 
C.P. 04510, M\'exico D.F.\\[-1.1mm]
\small 
\texttt{silva@iimas.unam.mx}
}
\date{}
\maketitle
\vspace{-4mm}
\begin{center}
  \textsl{To the memory of Professor Boris Pavlov}
\end{center}
  \centerline{{\bf Abstract}} \bigskip    
  On the basis of the explicit formulae for
  the action of the unitary group of exponentials corresponding to
  almost solvable extensions of a given closed symmetric operator with
  equal deficiency indices, we derive a new representation for the scattering matrix for
  pairs of such extensions.  We use this representation to {\it
    explicitly} recover the coupling constants in the inverse
  scattering problem for a finite non-compact quantum graph with
  $\delta$-type vertex conditions.
\section{Introduction}
\label{sec:introduction}
Over the last eighty years or so, the subject of the mathematical
analysis of waves interacting with obstacles and structures
(``scattering theory'') has served as one of the most impressive
examples of bridging abstract mathematics and physics applications,
which in turn motivated the development of new mathematical
techniques. The pioneering works of von Neumann \cite{MR1503053},
\cite{MR0066944} and his contemporaries during 1930--1950, on the mathematical foundations of quantum mechanics, fuelled the interest of
mathematical analysts to formulating and addressing the problems of
direct and inverse wave scattering in a rigorous way.

The foundations of the modern mathematical scattering theory were laid
by Friedrichs, Kato and 
Rosenblum \cite{MR0407617, MR0079235,
  MR0090028, Friedrichs} and subsequently by Birman and Kre\u\i n
\cite{MR0139007},  Birman \cite{Birman_1963}, Kato and Kuroda \cite{MR0385604} and Pearson
\cite{MR0328674}. For a detailed exposition of this subject, see
\cite{MR529429, MR1180965}. A parallel approach, which provides a
connection to the theory of dissipative operators, was developed by
Lax and Phillips \cite{MR0217440}, who analysed the direct scattering
problem for a wide class of linear operators in the Hilbert space,
including those associated with the multi-dimensional acoustic problem
outside an obstacle, using the language of group theory (and, indeed,
thereby developing the semigroup methods in operator theory). The
associated techniques were also termed ``resonance scattering'' by Lax
and Phillips.

By virtue of the underlying dissipative framework, the above activity
set the stage for the applications of non-selfadjoint techniques, such as the functional model for contractions and dissipative
operators by Sz\"{o}kefalvi-Nagy and Foia\c{s} \cite{MR2760647}, which showed the special r\^{o}le in it of the characteristic function of Liv\v{s}ic
\cite{MR0020719} and allowed Pavlov \cite{MR0510053} to construct a
spectral form of the functional model for dissipative operators.  The
connection between this work and the concepts of scattering theory was
uncovered by the famous theorem of Adamyan and Arov
\cite{MR0206711}. In a closely related development, Adamyan and Pavlov \cite{AdamyanPavlov}
established a description for the scattering matrix of a pair of self-adjoint extensions of a symmetric operator (densely or non-densely defined) with finite equal deficiency indices.

Further, Naboko \cite{MR573902} advanced the
research initiated by Pavlov, Adamyan and Arov in two
directions. Firstly, he generalised Pavlov's construction of the functional model in its spectral form 
to the case of non-dissipative operators, and secondly, he established its applicability to the scattering theory 
for pairs of non-selfadjoint operators.
 In particular, he provided
explicit formulae for the wave operators and scattering matrices of a
pair of (in general, non-selfadjoint) operators in the functional
model setting. It is remarkable that in this work of Naboko the
difference between the so-called stationary and non-stationary
scattering approaches disappears.

Our first aim in the present work is to discuss an extension of the approach of Naboko
\cite{MR573902}, which was formulated for additive perturbations of
self-adjoint operators, to the case of \emph{both self-adjoint and non-self-adjoint} extensions of symmetric
operators.  Our strategy is based on a version of the functional
model of Pavlov and Naboko as developed by Ryzhov \cite{MR2330831}.
The work \cite{MR2330831} stopped short of
proving the crucial, from the scattering point of view, theorem on
``smooth'' vectors and therefore was unable to extend Naboko's
results on the scattering theory to the setting of (in general,
non-selfadjoint) extensions of symmetric operators.

Our second aim is, using the above construction, to provide an explicit solution to an open problem of inverse scattering on a finite non-compact quantum graph, namely, the problem of determining matching conditions at the graph vertices. The uniqueness part of this problem has been treated in a preprint by Kostrykin and Schrader \cite{Kostrykin_Schrader}. There is also substantial literature on scattering for vector Schr\"{o}dinger operators on a half-line with matrix potentials, which corresponds to the particular case of a star-graph. Among the latest works on this subject we point out \cite{Weder2015}, \cite{Weder2016}, see also references therein, in which scattering is treated in the case of most general matching conditions at the vertex.

The mentioned problem on quantum graphs is a natural generalisation of
the classical problem of inverse scattering on the
infinite and semi-infinite line, which was solved using the classical
integral-operator techniques by Borg \cite{MR0015185, MR0058063},
Levinson \cite{MR0032067}, Krein \cite{MR0039895, MR0058072,
  MR0078543}, Gel'fand and Levitan \cite{MR0045281}, Marchenko
\cite{MR0075402}, Faddeev \cite{MR0149843, Faddeev_additional},  Deift and Trubowitz
\cite{MR0622619}.
This body of work has also included the
solution to the inverse spectral problem, {\it i.e.} the problem of
determining the potential in the Schr\"{o}dinger equation from the
spectral data. The inverse scattering problem in
these works is reduced to the analysis of the inverse problem based on
the Weyl-Titchmarsh $m$-coefficient, and our analysis below benefits from a reduction of the same kind.

In the operator-theoretic context, the $m$-coefficient is
generalised to both the classical Dirichlet-to-Neumann map (in the PDE
setting), and to the so-called $M$-operator, which takes the form of
the 
Weyl-Titchmarsh $M$-matrix in the case of quantum
graphs and, more generally, symmetric operators with finite deficiency
indices.  This generalisation has been exploited extensively in the study of
operators, self-adjoint and non-selfadjoint alike, through the works of
Krein's school in Ukraine on the theory of boundary triples and the
associated $M$-operators (Gorbachuk and Gorbachuk \cite{MR1294813},
Kochubei \cite{MR0365218, MR0592863}, Derkach and Malamud
\cite{MR1087947}). 
In our view, the theory of boundary triples is convenient for the
study of quantum graphs, when it can also be viewed as a version of
the celebrated Birman-Kre\v\i n-Vi\v{s}ik theory \cite{MR0080271,
  MR0024575, MR0052655}.

Quantum graphs, {\it i.e.} metric graphs with ordinary differential
operators acting on the edges subject to some ``coupling'' conditions
at the graph vertices, see {\it e.g.} \cite{MR3013208} are known to
combine one-dimensional and multidimensional
features. Assuming that the graph topology and the lengths of the
edges are known, for the operator of second differentiation on all
graph edges and $\delta$-type conditions at all graph vertices (see
Section~\ref{sec:quantum-graphs} for precise definitions), in the present paper we determine
the coupling constants at all vertices of a finite graph from the
knowledge of its scattering matrix.  Our approach to the above problem
uses as a starting point the strategy of the work \cite{MR2330831}
mentioned above, which derived the functional model for dissipative
restrictions of ``maximal'' operators, {\it i.e.} the adjoints of
symmetric densely-defined operators with equal deficiency indices. The
functional-model approach allows us to obtain a new formula for the
wave operators for any pair of such restrictions, in terms of the
$M$-operator for an appropriate boundary triple on the graph. This
formula, in turn, implies an expression for the scattering operator
and its spectral representation (``scattering matrix''). The obtained
formula is given explicitly in terms of the coupling constants at the
graph vertices, which allows us to carry out the inverse procedure of
recovering these constants from the knowledge of the scattering
matrix. Our approach is a development of the idea of Ershova {\it et al.}
\cite{MR3484377, MR3430381, MR3404107}, who studied the inverse spectral
problem and the inverse topology problem for quantum graphs using
boundary triples and $M$-operators.

The paper is organised as follows. In Section
\ref{sec:boundary-triples} we recall the key points of the theory of
boundary triples for extensions of symmetric operators with equal
deficiency indices and introduce the associated $M$-operators,
following mainly \cite{MR1087947} and \cite{MR2330831}. In Section \ref{strategy} we 
provide several observations that motivate the strategy of our analysis. 
In Section \ref{sec:functional-model} we recall 
the functional model for the above family of extensions and characterise the absolutely continuous subspace of $A_\varkappa$ as
the closure of the set of ``smooth'' vectors in the model Hilbert
space. On the basis
of this characterisation, in Section~\ref{sec:wave-operators} we
define the wave operators for a pair from the family $\{A_\varkappa\}$
and demonstrate their completeness property. This, in combination with
the functional model, allows us to obtain formulae for the scattering
operator of the pair (\emph{cf.} \cite{MR2386256}). In
Section~\ref{sec:spectral-repr-ac} we describe a convenient
representation of the scattering operator, namely the ``scattering
matrix'', which is explicitly written in terms of the $M$-operator,
analogous to the classical notion of the scattering matrix. All
material up to this point is applicable to a general class of
operators subject to the assumptions discussed in
Section~\ref{sec:boundary-triples}. 
In
Section~\ref{sec:quantum-graphs} we recall the concept of a quantum
graph and discuss the implications of the preceding theory for the
associated scattering operator for the pair $(A_\varkappa, A_0),$
where $\varkappa$ is the parametrising operator as before, now written
in terms of the ``coupling'' constants at the graph vertices and
$A_0=A_\varkappa\vert_{\varkappa=0}$ is the ``unperturbed'' operator
with Kirchhoff vertex conditions. Finally, in
Section~\ref{sec:inverse-scattering} we solve the inverse scattering
problem for a graph with $\delta$-type couplings at the vertices,
using the formulae for the scattering matrix in terms of the
$M$-matrix of the graph.


\section{Extension theory and boundary triples}
\label{sec:boundary-triples}

Let $\cH$ be a separable Hilbert space and denote by
$\inner{\cdot}{\cdot}$ the inner product in this space, which we
consider to be antilinear in the second argument.
Let $A$ be a closed symmetric operator densely defined in $\cH,$ {\it
  i.e.} $A\subset A^*,$ with domain $\dom(A)\subset\cH.$
For such operators, all points in the lower and
upper half-planes are of regular type with deficiency indices
\begin{equation*}
  n_\pm(A):=\dim(\cH\ominus\ran(A-z I))=\dim(\ker(A^*-\cc{z}I))\,,\quad z\in\complex_\pm\,.
\end{equation*}
If $A=A^*$ then $A$ is referred to as self-adjoint.
A closed
operator $L$ is said to be \emph{completely non-selfadjoint} if there
is no subspace reducing $L$ such that the restriction of $L$ to this subspace is self-adjoint. 
In this work we consider extensions of a given closed symmetric
operator $A$ with equal deficiency indices, {\it i.\,e.} $n_-(A)=n_+(A)$,
and use the theory of boundary triples.

In view of the importance of dissipative operators within the present work, 
we briefly recall that a densely defined operator $L$ in $\cH$
is called dissipative if
\begin{equation}
  \label{eq:definition-disipactive}
  \im\inner{Lf}{f}\ge 0\quad \quad\forall f\in\dom(L).
\end{equation}
 For a dissipative operator $L$, the lower half-plane is contained in the set of points of regular type, {\it i.e.}
\begin{equation*}
  \complex_-\subset\{z\in\complex: \exists C>0\ \ \forall f\in\dom(L)\ \ \norm{(L-zI)f}\ge
  C\norm{f}\}\,.
\end{equation*}
A dissipative operator $L$ is called maximal if $\complex_-$ is
actually contained in its resolvent set
$\rho(L):=\{z\in\complex: (L-zI)^{-1}\in\mathcal{B}(\cH)\}$.
($\mathcal{B}(\cH)$ denotes the space of bounded operators defined on
the whole Hilbert space $\cH$). Clearly, a maximal dissipative
operator is closed.

We next describe the boundary triple approach to
the extension theory of symmetric operators with equal deficiency
indices (see in \cite{Derkach} a review of the subject).
This approach is
particularly useful in the study of
self-adjoint extensions of differential operators of second order.


\begin{definition}
  \label{def:boudary-triple}
  For a closed symmetric operator $A$ with equal deficiency indices, consider the linear mappings
    $\Gamma_1:\dom(A^*)\to\mathcal{K},$
$\Gamma_0:\dom(A^*)\to\mathcal{K},$
where $\mathcal{K}$ is an auxiliary separable Hilbert space, such that
\begin{align}
(1)&\quad
  \inner{A^*f}{g}_\cH-\inner{f}{A^*g}_\cH=
\inner{\Gamma_1f}{\Gamma_0g}_\mathcal{K}-\inner{\Gamma_0f}{\Gamma_1g}_\mathcal{K};\label{Green_formula}\\[0.4em]
(2)&\quad \text{The mapping }\dom(A^*)\ni f\mapsto
\left(\begin{matrix}\Gamma_1f\\[0.3em]\Gamma_0f\end{matrix}\right)\in\mathcal{K}\oplus\mathcal{K}\text{ is surjective.}\nonumber
\end{align}
Then the triple $(\mathcal{K},\Gamma_1,\Gamma_0)$ is said to be a \emph{boundary
  triple} for $A^*$.
\end{definition}


In this work we consider \emph{almost solvable} extensions $A_B$ for which there exists a triple $(\mathcal K, \Gamma_1, \Gamma_0)$
and $B\in\mathcal{B}(\mathcal{K})$ 
such that
\begin{equation}
\label{eq:extension-by-operator}
  f\in\dom(A_B) \iff \Gamma_1f=B\Gamma_0f\,.
\end{equation}

The following assertions, written in slightly different terms, can be
found in \cite[Thm.\,2]{MR0365218} and \cite[Chap.\,3
Sec.\,1.4]{MR1154792} (see also \cite[Thm.\,2.3]{MR2732083},
\cite[Thm.\,1.1]{MR2330831}, and \cite[Sec.~14]{MR2953553} for a closer formulation). We compile them in the next
proposition for easy reference.  

\begin{proposition}
  \label{prop:properties-almost-extensions}
   Let $A$ be a closed symmetric operator with equal deficiency
  indices and let $(\mathcal{K},\Gamma_1,\Gamma_0)$ be a the boundary
  triple for $A^*$. Assume that $A_B$ is an almost solvable
  extension. Then the following statements hold:
  \begin{enumerate}
  \item $f\in\dom(A)$ if and only if $\Gamma_1f=\Gamma_0f=0.$
  \item $A_B$ is maximal, i.\,e., $\rho(A_B)\ne\emptyset$.
  \item $A_B^*=A_{B^*}.$
  \item $A_B$ is dissipative if and only if $B$ is dissipative.
  \item $A_B$ is self-adjoint if and only if $B$ is self-adjoint.
  \end{enumerate}
\end{proposition}

\begin{definition}
  \label{def:weyl-function}
  The function $M:\complex_-\cup\complex_+\to \mathcal{B}(\cH)$ such
  that
  \begin{equation*}
    M(z)\Gamma_0f=\Gamma_1f\ \ \ \ \ \forall f\in\ker(A^*-zI)
  \end{equation*}
is \emph{the Weyl function of the boundary triple}
  $(\mathcal{K},\Gamma_1,\Gamma_0)$ for $A^*,$ where $A$ is assumed to be as in Proposition \ref{prop:properties-almost-extensions}.
 \end{definition} 

The Weyl function defined above has the following properties \cite{MR1087947}.
\begin{proposition}
  \label{prop:weyl-function-properties}
  Let $M$ be a Weyl function of the boundary triple
  $(\mathcal{K},\Gamma_1,\Gamma_0)$ for $A^*,$ where $A$ is a closed symmetric operator with equal deficiency indices. Then the following statements hold:
  \begin{enumerate}
  \item $M:\complex\setminus\reals\to\mathcal{B}(\mathcal{K})$\,.
  \item $M$ is a $\mathcal{B}(\mathcal{K})$-valued double-sided
    $\mathcal{R}$-function \cite{MR0328627KK}, that is,
    \begin{equation*}
      M(z)^*=M(\cc{z})\quad\text{ and }\quad\im(z)\im\bigl(M(z)\bigr)>0\quad\text{ for }
      z\in\complex\setminus\reals\,.
    \end{equation*}
  \item The spectrum of $A_B$ coincides with the set of points $z_0\in{\mathbb C}$ such
    that $(M-B)^{-1}$
    does not admit analytic continuation into $z_0.$
  \end{enumerate}
\end{proposition}


We next lay out the notation for some of the main objects in our analysis.  
In the auxiliary Hilbert space $\mathcal{K}$, choose
  a bounded positive self-adjoint operator $\alpha$ so
  that the operator
 \begin{equation}
    \label{eq:b-kappa-def}
    B_\varkappa:=\frac{\alpha\varkappa\alpha}{2}
  \end{equation}
 belongs to $\mathcal{B}(\mathcal{K})$, where $\varkappa$ is a bounded operator\footnote{Clearly, the assumption that ${\rm ker}(\alpha)=\{0\}$ is without loss of generality, by a suitable modification of $\varkappa$ if necessary.} in $\mathcal{K}.$
In what follows, 
we deal with almost solvable extensions of a given
symmetric operator $A$ that are generated by $B_\varkappa$ via (\ref{eq:extension-by-operator}).
It is always assumed that the deficiency indices of $A$ are equal and that some boundary triple
$(\mathcal{K},\Gamma_1,\Gamma_0)$ for
$A^*$ is fixed.
In order to streamline the
 formulae, we write
 \begin{equation}
   \label{eq:a-kappa-def}
   A_\varkappa:=A_{B_\varkappa}.
 \end{equation}
Here $\varkappa$ should be understood as a parameter for a family of
almost solvable extensions of $A$. Note that if $\varkappa$ is
self-adjoint then so is $B_\varkappa$ and, hence by
Proposition~\ref{prop:properties-almost-extensions}(5), $A_\varkappa$ is
self-adjoint. Note also that $A_{\I I}$ is maximal dissipative, again by
Proposition~\ref{prop:properties-almost-extensions}.

\begin{definition}
  \label{def:charasteristic-function}
  The characteristic function of the operator $A_{\I I}$ is
  the operator-valued function $S$ on $\complex_+$ given by
  \begin{equation}
    S(z):=I+\I\alpha\bigl(B_{\I I}^*-M(z)\bigr)^{-1}\alpha,\ \ \ \ \ z\in\complex_+.
    \label{S_definition}
  \end{equation}
 \end{definition}
 \begin{remark}
  \label{rem:def-charasteristic-function}
  The function $S$ is analytic in $\complex_+$ and, for each
  $z\in\complex_+$, the mapping $S(z):{\mathcal K}\to {\mathcal K}$ is a contraction. Therefore, $S$ has
  nontangential limits almost everywhere on the real line in the
  strong topology \cite{MR2760647},
  which we henceforth denote by $S(k),$ $k\in{\mathbb R}.$
\end{remark}
\begin{remark}
When
  $\alpha=\sqrt{2}I,$ which is the case of our application to finite quantum graphs, a straightforward calculation yields that
  $S(z)$ is the Cayley transform of $M(z)$, {\it i.e.}
 \begin{equation}
   S(z)=\bigl(M(z)-\I I\bigr)\bigl(M(z)+\I I\bigr)^{-1}\,.
   \label{SM_formula}
 \end{equation}
\end{remark}

\section{General remarks on our approach}
\label{strategy}

Our approach to mathematical scattering theory for extensions of closed symmetric operators (direct and inverse) will be based 
on the 
functional model for a family of almost solvable extensions of the given minimal symmetric operator. Our choice of this method 
is based on the following considerations:

1) We would like to consider scattering problems where at least one of the two operators of the pair is non-selfadjoint. In contrast, the classical scattering results only pertain to pairs of self-adjoint operators: even the definition of $\exp({\rm i}Lt)$ in the case of non-selfadjoint $L$ needs to be clarified. While there are various ways to construct functional calculus for non-uniformly bounded groups, say the Riesz-Dunford calculus, the most attractive of them for us 
is via developing a functional model where the exponent is represented by an operator of multiplication on a linear set dense in the absolutely continuous subspace of its generator. 
 The ``symmetric Pavlov model'' \cite{MR0510053}, which we describe in Section \ref{sec:functional-model},  provides such an approach.




2) Naboko \cite{MR573902} has shown how to construct mathematical scattering for a class of non-selfadjoint operators in the ``additive'' case $L=A+{\rm i}V,$ where $A,$ $V$ are self-adjoint operators. 
This kind of method offers some advantages in comparison with other techniques: a) the difference between stationary and non-stationary theories disappears, in the sense that the same construction yields explicit expressions for both; b) the scattering operator is represented in a concise form, which, in particular, immediately yields a formula for the spectral representation in the eigenfunction basis of the unperturbed self-adjoint operator (``scattering matrix"). 


3)   It has to be pointed out that the seemingly non-selfadjoint approach due to Naboko contains the self-adjoint setting as its particular case, and when applied this way it yields all the classical results
 ({\it e.g.} Pearson Theorem, Birman-Krein-Kuroda Theorem, as well as their generalisations).
 In this self-adjoint setting this approach proves to be consistent with the ``smooth" scattering theory (see \cite{MR1180965}). As in the case of the latter, the principal r\^{o}le in Naboko's construction is played by a linear dense subset of the absolutely continuous subspace (``smooth'' vectors), which in the self-adjoint case is described by the so-called Rosenblum Lemma \cite{MR0090028}. In the non-self-adjoint case the corresponding linear dense subset is identified by the property that the resolvent acts on it as the resolvent of the operator of multiplication in the symmetric Pavlov representation, {\it cf.} (\ref{New_Representation}). This, in turn, facilitates the derivation of explicit formulae for wave operators on these dense sets of smooth vectors. The construction of the wave operators is then completed by passing to a closure.

 In what follows we briefly describe the approach introduced above and the results obtained on this way, essentially building up on the earlier results pertaining to the analysis of non-self-adjoint extensions, due to Ryzhov. These allow us to generalise Naboko's construction of wave operators and scattering matrices to the case studied in the present paper.  In order to deal with the family of extensions $\{A_\varkappa\}$ of
the operator $A$ (where the parameter $\varkappa$ is itself an operator, see notation
immediately following Proposition \ref{prop:weyl-function-properties}), we first construct a functional model of its
particular dissipative extension. This is done following the
Pavlov-Naboko procedure, which in turn stems from Sz.-Nagy-Foia\c{s}
functional model. This allows us to obtain a simple model for the
whole family $\{A_\varkappa\},$ in particular yielding a possibility
to apply it to the scattering theory for certain pairs of operators in
$\{A_\varkappa\}$, including both the cases when these operators are
self-adjoint and non-selfadjoint.
 In view of transparency, we try to reduce the technicalities to the bare minimum, at the same time pointing out that the corresponding complete proofs of the necessary statements can be found in \cite{ChKS_OTAA}.




\section{Functional model}
\label{sec:functional-model}

Following \cite{MR573902}, we introduce a Hilbert space
serving as a functional model for the family of operators $A_\varkappa$. 
 This functional model was constructed for completely non-selfadjoint maximal
dissipative operators in \cite{MR0510053, MR0365199, Drogobych} and further
developed in \cite{MR573902}. Next we recall some related necessary
information. In what follows, in various formulae, we use the
subscript ``$\pm$'' to indicate two different versions of the same formula in which the subscripts ``$+$'' and ``$-$'' are taken individually.

A ${\mathcal K}$-valued function $f,$ analytic on $\complex_\pm,$  is said to be
in the Hardy class $H^2_\pm({\mathcal K})$ if ({\it cf.} \cite[Sec.\,4.8]{MR822228})
\begin{equation*}
  \sup_{y>0}\int_\reals\bigl\Vert f(x\pm\I y)\bigr\Vert_{\mathcal K}^2dx<+\infty.
\end{equation*} 
Whenever $f\in H^2_\pm({\mathcal K})$, the left-hand side of the above inequality defines $\norm{f}_{H^2_\pm({\mathcal K})}^2$. 
We use the notation $H^2_+$ and $H^2_-$ for the usual Hardy spaces of ${\mathbb C}$-valued functions. 
Any element in the Hardy spaces $H^2_\pm({\mathcal K})$ can be associated with
its boundary values in the topology of ${\mathcal K},$ which exist almost
everywhere on the real line. The spaces of boundary functions of
$H_\pm^2({\mathcal K})$ are denoted by $\widehat{H}_\pm^2({\mathcal K}),$ and they are subspaces
of $L^2(\reals,{\mathcal K})$ \cite[Sec.\,4.8, Thm.\,B]{MR822228}).  By the
Paley-Wiener theorem \cite[Sec.\,4.8, Thm.\,E]{MR822228}), these subspaces are the orthogonal complements of each
other ({\it i.e.}, $L^2(\reals,{\mathcal K})=\widehat{H}^2_+({\mathcal K})\oplus \widehat{H}^2_-({\mathcal K})$).


As mentioned above, 
the characteristic function $S$ 
has non-tangential limits
almost everywhere on the real line in the strong topology. Thus, for a
two-component vector function $\binom{\widetilde{g}}{g}$ taking values
in ${\mathcal K}\oplus {\mathcal K}$, one can consider the integral
\begin{equation}
  \label{eq:inner-in-functional}
  \int_\reals\inner{\begin{pmatrix} I & S^*(s)\\
    S(s) &
    I\end{pmatrix}\binom{\widetilde{g}(s)}{g(s)}}{\binom{\widetilde{g}(s)}{g(s)}}_{{\mathcal K}\oplus
{\mathcal K}}ds,
\end{equation}
which is always nonnegative, due to the contractive properties of $S.$
The space
\begin{equation}
\label{mathfrakH}
\mathfrak{H}:=L^2\Biggl({\mathcal K}\oplus {\mathcal K}; \begin{pmatrix}
    I & S^*\\
    S & I
  \end{pmatrix}\Biggr)
\end{equation}
is the completion of the linear set of two-component vector functions
$\binom{\widetilde{g}}{g}: {\mathbb R}\to {\mathcal K}\oplus {\mathcal K}$ in the norm
(\ref{eq:inner-in-functional}), factored with respect to vectors of zero norm.
Naturally, not every element of the set can be identified with a pair
$\binom{\widetilde{g}}{g}$ of two independent functions. Still, in
what follows we keep the notation $\binom{\widetilde{g}}{g}$ for the
elements of this space.

  Another consequence of the contractive properties of the
  characteristic function $S$ is that for $\widetilde{g},g\in
  L^2(\reals,{\mathcal K})$ one has
  \begin{equation*}
    \norm{\binom{\widetilde{g}}{g}}_\mathfrak{H}\ge
    \begin{cases}
      \norm{\widetilde{g}+S^*g}_{L^2(\reals,{\mathcal K})},\\[0.4em]
      \norm{S\widetilde{g}+g}_{L^2(\reals,{\mathcal K})}\,.
    \end{cases}
  \end{equation*}
  Thus, for every Cauchy sequence
  $\{\binom{\widetilde{g}_n}{g_n}\}_{n=1}^\infty$, with respect to the $\mathfrak{H}$-topology,
  such that $\widetilde{g}_n,g_n\in L^2(\reals,{\mathcal K})$ for all
  $n\in\nats$, the limits of $\widetilde{g}_n+S^*g_n$ and
  $S\widetilde{g}_n+g_n$ exists in $L^2(\reals,{\mathcal K})$, so that the
  objects $\widetilde{g}+S^*g$ and
  $S\widetilde{g}+g$ can always be treated as $L^2(\reals,{\mathcal K})$ functions.

Furthermore, consider the orthogonal subspaces of $\mathfrak{H}$
\begin{equation}
  \label{eq:D-spaces}
  D_-:=
  \begin{pmatrix}
    0\\[0.3em]
    \widehat{H}^2_-({\mathcal K})
  \end{pmatrix}\,,\quad
 D_+:=
  \begin{pmatrix}
   \widehat{H}^2_+({\mathcal K})\\[0.3em]
   0
  \end{pmatrix},
\end{equation}
and define the space
$K:=\mathfrak{H}\ominus(D_-\oplus D_+),$
which is characterised as follows, see {\it e.g.} \cite{MR0365199, Drogobych}:
\begin{equation}
  K=\left\{\begin{pmatrix}
   \widetilde{g}\\
   g
  \end{pmatrix}\in\mathfrak{H}: \widetilde{g}+S^*g\in  \widehat{H}^2_-({\mathcal K})\,,
  S\widetilde{g}+g\in
 \widehat{H}^2_+({\mathcal K})\right\}\,.
 \label{characterise_K}
\end{equation}
The orthogonal projection $P_K$ onto the subspace
$K$ is given by (see {\it e.g.} \cite{MR0500225})
\begin{equation}
\label{eq:pk-action}
  P_K
  \begin{pmatrix}
    \widetilde{g}\\
    g
  \end{pmatrix}
=
\begin{pmatrix}
  \widetilde{g}-P_+(\widetilde{g}+S^*g)\\[0.3em]
  g-P_-(S\,\widetilde{g}+g)
\end{pmatrix}\,,
\end{equation}
where $P_\pm$ are the orthogonal Riesz projections in $L^2({\mathcal K})$ onto
$\widehat{H}^2_\pm({\mathcal K})$.

A completely
non-selfadjoint dissipative operator admits \cite{MR2760647} a
self-adjoint dilation. 
The dilation $\mathscr{A}=\mathscr{A}^*$ of the operator $A_{\I I}$ is constructed following Pavlov's
procedure \cite{MR0365199, MR0510053, Drogobych}:  it is defined in the Hilbert space
  $\mathscr{H}=
  L^2(\reals_-,{\mathcal K})\oplus\mathcal{H}\oplus L^2(\reals_+,{\mathcal K}),
  $
so that
\begin{equation*}
 P_\mathcal{H}(\mathscr{A}-zI)^{-1}\eval{\mathcal{H}}=
 (A_{\I I}-z I)^{-1}\,, \qquad z\in\complex_-.
\end{equation*}
As in the case of additive non-selfadjoint perturbations
\cite{MR573902}, Ryzhov established in
\cite[Thm.\,2.3]{MR2330831} that $\mathfrak{H}$ serves as the
functional model for the dilation $\mathscr{A}$ {\it i.e.} there exists
an isometry $\Phi: \mathscr{H}\to\mathfrak{H}$ such that
$\mathscr{A}$ is transformed into the operator of multiplication by
the independent variable: 
 $ \Phi(\mathscr{A}-z I)^{-1}=(\cdot-z)^{-1}\Phi\,.$
Furthermore, under this isometry 
  $$
  \Phi\eval\cH\cH=K
  $$
   unitarily, where $\cH$ is understood as being embedded in $\mathscr{H}$ in the natural way, {\it i.e.} 
   $$
   \cH\ni h\mapsto 0\oplus h\oplus0\in\mathscr{H}.
   $$
  In what follows we keep the label $\Phi$ for the restriction $\Phi\eval\cH,$ in hope that it does not lead to confusion.


Following the ideas of Naboko, in the functional model space $\mathfrak{H}$
consider two subspaces 
\begin{equation}
   \mathfrak{N}^\varkappa_\pm:=\left\{\binom{\widetilde{g}}{g}\in\mathfrak{H}:
  P_\pm\left(\chi_\varkappa^+(\widetilde{g}+S^*g)+\chi_\varkappa^-(S\widetilde{g}
  +g)\right)=0\right\},
  \label{definition_curlyN}
\end{equation}
where
\begin{equation*}
  \chi_\varkappa^\pm:=\frac{I\pm\I\varkappa}{2}.
\end{equation*}
These subspaces have a characterisation in terms of the resolvent of the operator $A_\varkappa,$ whose proof, see \cite{ChKS_OTAA}, follows the approach of \cite[Thm.\,4]{MR573902}.
\begin{theorem}[\cite{ChKS_OTAA}]
  \label{lem:similar-to-naboko-thm-4}
The following characterisation holds:
\begin{equation}
 \mathfrak{N}^\varkappa_\pm=\left\{\binom{\widetilde{g}}{g}\in\mathfrak{H}:
  \Phi(A_{\varkappa}-z I)^{-1}\Phi^*P_K\binom{\widetilde{g}}{g}=
P_K\frac{1}{\cdot-z}\binom{\widetilde{g}}{g}
\text{ for all } z\in\complex_\pm\right\}\,.
\label{curlyN_plusminus}
\end{equation}
\end{theorem}

Consider also the counterparts of $\mathfrak{N}^\varkappa_\pm$ in the original
Hilbert space $\cH:$
\begin{equation}
\label{eq:definition-n_pm}
  \widetilde{N}_\pm^\varkappa:=\Phi^*P_K\mathfrak{N}^\varkappa_\pm\,,
\end{equation}
which are linear sets albeit not necessarily subspaces.
In a way similar to \cite{MR573902}, we define the set
\begin{equation*}
  \widetilde{N}_{\rm e}^\varkappa:=\widetilde{N}_+^\varkappa\cap
 \widetilde{N}_-^\varkappa
\end{equation*}
of so-called smooth vectors and its closure $N_{\rm e}^\varkappa:=\clos(\widetilde{N}_{\rm e}^\varkappa).$ 
This proves to be suitable for the description of
the absolutely continuous subspace and, therefore, for the construction
of the wave operators. 


\begin{definition}
\label{abs_cont_subspace}
For a symmetric operator $A,$ in the case of a non-selfadjoint extension $A_\varkappa$ the absolutely continuous subspace $\cH_{\rm ac}(A_{\varkappa})$  is defined by the formula $\cH_{\rm ac}(A_{\varkappa})=N^\varkappa_{\rm e}.$
\end{definition}
The next statement, the proof of which is given in \cite{ChKS_OTAA}, 
motivates the above definition.

\begin{theorem}[Self-adjoint case, see \cite{ChKS_OTAA}]
  \label{thm:on-smooth-vectors-a.c.equality}
  Assume that $\varkappa=\varkappa^*$  (equivalently, $A_\varkappa=A_\varkappa^*$)
  and let  $\alpha\Gamma_0(A_{\varkappa}-z I)^{-1}$ be a Hilbert-Schmidt
  operator for at least one point $z\in\rho(A_\varkappa)$. If
  $A$ is completely non-selfadjoint, then
    $N_{\rm e}^\varkappa = \cH_{\rm ac}(A_{\varkappa}).$
\end{theorem}
Definition \ref{abs_cont_subspace} follows in the footsteps of the corresponding definition by Naboko \cite{MR573902} in the case of additive perturbations. In particular, an argument similar to \cite[Corollary 1]{MR573902} shows that for the functional model image of $\widetilde{N}^\varkappa_{\rm e}$ the following representation holds:
\begin{align}
&\Phi\widetilde{N}^\varkappa_{\rm e}=\biggl\{P_K\binom{\widetilde{g}}{g}:\nonumber\\
&\binom{\widetilde{g}}{g}\in\mathfrak{H}\ {\rm satisfies}\ \Phi(A_{\varkappa}-z I)^{-1}\Phi^*P_K\binom{\widetilde{g}}{g}=
P_K\frac{1}{\cdot-z}\binom{\widetilde{g}}{g}\ \ \ \forall\,z\in{\mathbb C}_-\cup{\mathbb C}_+\biggr\}.
\label{New_Representation}
\end{align}
(Note that the inclusion of the right-hand side of (\ref{New_Representation}) into $\Phi\widetilde{N}^\varkappa_{\rm e}$ follows immediately from Theorem \ref{lem:similar-to-naboko-thm-4}.) Further, we arrive at an equivalent description ({\it cf.} (\ref{definition_curlyN})):
  \begin{equation}
    \Phi \widetilde{N}_{\rm e}^\varkappa=\left\{P_K\binom{\widetilde{g}}{g}: \binom{\widetilde{g}}{g}\in\mathfrak{H}
    \ {\rm satisfies}\ 
 \chi_\varkappa^+(\widetilde{g}+S^*g)+\chi_\varkappa^-(S\widetilde{g}+g)=0\right\}\,.
 \label{lyubimaya_formula}
  \end{equation}

The representations (\ref{New_Representation}), (\ref{lyubimaya_formula}) illustrate the r\^{o}le of the 
subspace of smooth vectors as the subspace in whose image under the isometry $\Phi$ the 
operator $A_\varkappa$ acts as multiplication by the independent variable. This property is crucial in the derivation of the formulae for wave operators of pairs from the family $\{A_\varkappa\},$ which we present in the next section and which are subsequently used in the solution of the inverse problem for quantum graphs in Sections \ref{sec:quantum-graphs}, \ref{sec:inverse-scattering}.

\section{Wave and scattering operators}
\label{sec:wave-operators}
The results discussed above allow us to calculate the wave
operators for any pair $A_{\varkappa_1},A_{\varkappa_2}$, where
$A_{\varkappa_1}$ and $A_{\varkappa_2}$ are operators in the class
introduced in Section~\ref{sec:boundary-triples}. 
For simplicity, and bearing in mind the
application of the abstract construction to the problem described in
Sections \ref{sec:quantum-graphs} and \ref{sec:inverse-scattering}, in what follows we set $\varkappa_2=0$ and write
$\varkappa$ instead of $\varkappa_1$. Note that $A_0$ is a self-adjoint
operator, which is convenient for presentation purposes.

We begin 
by recalling the model representation for the function
$\exp(iA_\varkappa t)$, $t\in\reals$, of the operator $A_\varkappa,$ evaluated on the set of smooth vectors 
$\widetilde{N}_{\rm e}^\varkappa,$ as well as a proposition describing such vectors in
 $\widetilde{N}_{\rm e}^\varkappa$ and $\widetilde{N}_{\rm e}^0$ that the difference between their respective dynamics vanishes as $t\to-\infty.$
\begin{proposition}(\cite[Prop.\,2]{MR573902})
  \label{prop:exp-funtion}
  For all $t\in\reals$ and all $\binom{\widetilde{g}}{g}$ such that
  $\Phi^*P_K\binom{\widetilde{g}}{g}\in\widetilde{N}_{\rm e}^\varkappa$ one has
  \begin{equation*}
    \Phi\exp(iA_\varkappa t)\Phi^*
  P_K\binom{\widetilde{g}}{g}=
P_K\exp(ikt)\binom{\widetilde{g}}{g}.
  \end{equation*}
\end{proposition}

\begin{proposition}(\cite[Section 4]{MR573902})
  \label{prop:pre-wave-op}
  If $\Phi^*P_K\binom{\widetilde{g}}{g}\in\widetilde{N}_{\rm e}^\varkappa$
  and $\Phi^*P_K\binom{\widetilde{g}'}{g}\in\widetilde{N}_{\rm e}^0$
  (with the same element\footnote{Despite the fact that
    $\binom{\widetilde{g}}{g}\in\mathfrak{H}$ is nothing but a symbol,
    still $\widetilde{g}$ and $g$ can be identified with vectors in
    certain $L^2({\mathcal K})$ spaces with  operators ``weights'', see
    details below in Section~\ref{sec:spectral-repr-ac}. Further, we recall that even then
    for $\binom{\widetilde{g}}{g}\in\mathfrak{H}$, the components $\widetilde{g}$ and $g$ are not, in general, \emph{independent} of each other.} $g$), then
\begin{equation*}
  \norm{\exp(-iA_\varkappa t)\Phi^*P_K\binom{\widetilde{g}}{g}-\exp(-iA_0
  t)\Phi^* P_K\binom{\widetilde{g}'}{g}}_{\mathfrak H}\convergesto{t}{-\infty}0.
\end{equation*}
\end{proposition}

It follows from Proposition \ref{prop:pre-wave-op}  that whenever $\Phi^*P_K\binom{\widetilde{g}}{g}\in\widetilde{N}_{\rm e}^\varkappa$
and $\Phi^*P_K\binom{\widetilde{g}'}{g}\in\widetilde{N}_{\rm e}^0$ (with the same second component $g$), one formally has
\begin{equation*}
\lim_{t\to-\infty}e^{iA_0t}e^{-iA_\varkappa t}\Phi^*P_K\binom{\widetilde{g}}{g}=\Phi^*P_K\binom{\widetilde{g}'}{g}
=\Phi^*P_K\binom{-(I+S)^{-1}(I+S^*)g}{g}\,,
\end{equation*}
where in the last equality we use the inclusion $\Phi^*P_K\binom{\widetilde{g}'}{g}\in\widetilde{N}_{\rm e}^0,$ which by (\ref{lyubimaya_formula}) yields $\widetilde{g}'+S^*g+S\widetilde{g}'+g=0.$

In what follows we use the standard definition of wave operators,
see {\it e.g.} \cite{MR0407617}, allowing the operator $A_\varkappa$ to be non-selfadjoint:
\begin{equation}
  W_\pm(A_0, A_\varkappa):=\slim_{t\to\pm\infty}e^{iA_0t}e^{-iA_\varkappa t}P_{\rm ac}^\varkappa,\qquad 
  W_\pm(A_\varkappa, A_0):=\slim_{t\to\pm\infty}e^{iA_\varkappa t}e^{-iA_0 t}P_{\rm ac}^0.
\label{wave_def}
\end{equation}
In the above formulae, 
we denote by  $P_{\rm ac}^\varkappa,$ $P_{\rm ac}^0$ the projections onto the absolutely continuous subspace of $A_\varkappa,$ see Definition \ref{abs_cont_subspace}, and the absolutely continuous subspace of the self-adjoint operator $A_0,$ respectively.

It follows that for $\Phi^*P_K\binom{\widetilde{g}}{g}\in\widetilde{N}_{\rm e}^\varkappa$ one has
\begin{equation}
  \label{eq:formula-0-kappaw-}
  W_-(A_0,A_\varkappa)\Phi^*P_K\binom{\widetilde{g}}{g}=\Phi^*P_K\binom{-(I+S)^{-1}(I+S^*)g}{g}\,.
\end{equation}
One argues in a similar way in the case of the wave operator $W_+(A_0,A_\varkappa),$ as well as in the case of the wave operators 
$W_\pm(A_\varkappa,A_0),$ which we  
define by
\begin{equation*}
  \norm{e^{-iA_\varkappa
      t}W_\pm(A_\varkappa,A_0)\Phi^*P_K\binom{\widetilde{g}}{g}-e^{-iA_0t}\Phi^*P_K\binom{\widetilde{g}}{g}
}_{\mathfrak H}\convergesto{t}{\pm\infty}0,\qquad \Phi^*P_K\binom{\widetilde{g}}{g}\in\widetilde{N}_{\rm e}^0.
\end{equation*}

\begin{theorem}[\cite{ChKS_OTAA}]
  \label{thm:existence-completeness-wave-operators}
  Let $A$ be a closed, symmetric, completely non-selfadjoint operator
  with equal deficiency indices and consider its extension $A_\varkappa,$ as described in Section \ref{sec:boundary-triples}, under the assumption that $A_\varkappa$ has at least one regular point in ${\mathbb C}_+$ and in ${\mathbb C}_-.$ 
  If $S-I$ is compact 
  in $\overline{\mathbb C}_+,$  
   then the wave operators
  $W_\pm(A_0,A_\varkappa)$ exist on dense sets in $N_{\rm e}^\varkappa$ and for all $\Phi^*P_K\binom{\widetilde{g}}{g}\in\widetilde{N}_{\rm e}^\varkappa$ one has (\ref{eq:formula-0-kappaw-}) and 
  \begin{equation}
W_+(A_0,A_\varkappa)\Phi^*P_K\binom{\widetilde{g}}{g}
=\Phi^*P_K\left(\begin{matrix}\widetilde{g}\\[0.3em]-(I+S^*)^{-1}(I+S)\widetilde{g}\end{matrix}\right).
\label{eq:formula-0-kappaw+}
\end{equation}
Similarly, the wave operators and $W_\pm(A_\varkappa,A_0)$  exist on dense sets in $\mathcal{H}_{\rm ac}(A_0)$ and for all $\Phi^*P_K\binom{\widetilde{g}}{g}\in\widetilde{N}_{\rm e}^0$ one has 
\begin{align}
  \label{eq:formula-kappa-0w-}
 &W_-(A_\varkappa,A_0)\Phi^*P_K\binom{\widetilde{g}}{g}=\Phi^*P_K\left(\begin{matrix}-\bigl(I+\chi_\varkappa^-(S-I)\bigr)^{-1}\bigl(I+\chi_\varkappa^+(S^*-I)\bigr)g\\[0.3em] g\end{matrix}\right),
 \\[1.0em]
 \label{eq:formula-kappa-0w+}
  &W_+(A_\varkappa,A_0)\Phi^*P_K\binom{\widetilde{g}}{g}=\Phi^*P_K
\left(\begin{matrix}\widetilde{g}\\[0.3em]
-\bigl(I+\chi_\varkappa^+(S^*-I)\bigr)^{-1}
\bigl(I+\chi_\varkappa^-(S-I)\bigr)\widetilde{g}\end{matrix}\right),
\end{align}
The ranges of $W_\pm(A_0,A_\varkappa)$ and  $W_\pm(A_\varkappa,A_0)$ are dense in $\mathcal{H}_{\rm ac}(A_0)$ and $N_{\rm e}^\varkappa,$ respectively.\footnote{In the case when $A_\varkappa$ is self-adjoint, or, in general, the named wave operators are bounded, the claims of the theorem are equivalent (by the classical Banach-Steinhaus theorem) to the statement of the existence and completeness of the wave operators for the pair $A_0, A_\varkappa.$ Sufficient conditions of boundedness of these wave operators are contained in {\it e.g.} \cite[Section 4]{MR573902}, \cite {MR1252228} and references therein.}
\end{theorem}

\begin{proof}[Sketch of the proof]
In order to rigorously justify the above formal derivation of \eqref{eq:formula-0-kappaw-}--\eqref{eq:formula-kappa-0w+}, {\it i.e.} in order to prove the existence and completeness of the wave operators, one needs to show that the right-hand sides of the formulae \eqref{eq:formula-0-kappaw-}--\eqref{eq:formula-kappa-0w+} make sense on dense subsets of the corresponding absolutely continuous subspaces.
Noting that  \eqref{eq:formula-0-kappaw-}--\eqref{eq:formula-kappa-0w+} have the form identical to the expressions for wave operators derived in \cite[Section 4]{MR573902}, \cite {MR1252228}, this justification is an appropriate modification of the argument of \cite {MR1252228}.

Indeed, consider the formula (\ref{eq:formula-0-kappaw-}). Here one needs to attribute a correct sense to the expression $(I+S(k))^{-1}$ a.e. on the real line. Let $S(z)-I$ satisfy the assumption of the theorem. Then, using the strategy of \cite{MR1252228}, one has non-tangential boundedness of $(I+S(z))^{-1}$ for almost all points of the real line.
On the other hand, the latter inverse can be computed in $\mathbb C_+$:
\begin{equation*}
\bigl(I+S(z)\bigr)^{-1}=\frac{1}{2}\bigl(I+ i \alpha M(z)^{-1}\alpha/2\bigr).
\end{equation*}
It follows from the analytic properties of $M(z)$ that the inverse $(I+S(z))^{-1}$ exists everywhere in the upper half-plane. Thus, \cite{MR1252228} yields that $(I+S(z))^{-1}$ is $\mathbb R$-a.e. nontangentially bounded and it admits measurable non-tangential limits in the strong operator topology almost everywhere on $\mathbb R$. As it is easily seen, these limits must then coincide with $(I+ S(k))^{-1}$ for almost all $k\in \mathbb R$.

The presented argument allows one to verify the correctness of the formula (\ref{eq:formula-0-kappaw-}).
Indeed, consider $\mathbbm{1}_n(k),$ the indicator of the set $\{k\in \mathbb R: \|(I+S(k))^{-1}\|\leq n\}.
$
Clearly, $\mathbbm{1}_n(k)\to 1$ as $n\to \infty$ for almost all $k\in \mathbb R$. Next, suppose that $P_K(\widetilde g, g)\in \widetilde N_{\rm e}^\varkappa$. Then $P_K\mathbbm{1}_n(\widetilde g, g)$ is also a smooth vector and
$$
\binom{-(I+S)^{-1}\mathbbm{1}_n (I+S^*) g}{\mathbbm{1}_n g}\in \mathfrak H.
$$
It follows, by the Lebesgue dominated convergence theorem, that the set of vectors $P_K\mathbbm{1}_n(\widetilde g, g)$ is dense in $N_{\rm e}^\varkappa$. 

The remaining three wave operators are treated in a similar way, see the complete details in \cite{ChKS_OTAA}. Finally, the density of the range of the four wave operators follows from the density of their domains, by a standard inversion argument, see {\it e.g.} \cite{MR1180965}.
\end{proof}

\begin{remark}[\cite{ChKS_OTAA}]
\label{long_remark}

1. The condition of the above theorem that $S(z)-I$
is compact 
in $\overline{\mathbb C}_+$  is satisfied \cite{Krein, ChKS_OTAA}, as long as the scalar function $\|\alpha M(z)^{-1}\alpha\|_{\mathfrak S_p}$ is nontangentially bounded almost everywhere on the real line for some $p<\infty,$ where ${\mathfrak S_p},$
$p\in(0, \infty],$ are the standard Schatten -- von Neumann classes of compact operators.

2. An alternative sufficient condition is the condition $\alpha\in\mathfrak S_2$ (and therefore $B_\varkappa \in \mathfrak S_1$), or, more generally, $\alpha M(z)^{-1}\alpha\in\mathfrak S_1,$ see \cite{MR1036844} for details.

3. Following from the analysis above, the existence and completeness of the wave operators for the par $A_\varkappa,$ $A_0$ is closely linked to the condition of $\alpha$ having a ``relative Hilbert-Schmidt property'' with respect to $M(z).$
Recalling that $B_\varkappa=\alpha\varkappa\alpha/2,$ this is not always feasible to expect. Nevertheless, by appropriately modifying the boundary triple, the situation can often be rectified. For example, if  $C_\varkappa=C_0+\alpha\varkappa\alpha/2,$ where $C_0$ and  $\varkappa$ are bounded and $\alpha\in{\mathfrak S}_2,$ replaces the operator $B_\varkappa$ in (\ref{eq:b-kappa-def}), then one ``shifts" the boundary triple:
$\widehat{\Gamma}_0=\Gamma_0,$ $\widehat{\Gamma}_1=\Gamma_1-C_0\Gamma_0.$ One thus obtains that in the new triple $({\mathcal K}, \widehat{\Gamma}_0, \widehat{\Gamma}_1)$ the operator $A_{\varkappa}$  coincides with the extension  corresponding to the boundary operator $B_\varkappa=\alpha\varkappa\alpha/2,$ whereas the Weyl-Titchmarsh function $M(z)$ undergoes a shift to the function $M(z)-C_0.$ The proof of Theorem 6.1 remains intact, while Part 2 of this remark yields that the condition $\alpha (M(z)-C_0)^{-1}\alpha\in{\mathfrak S}_1$ guarantees the existence and completeness of the wave operators for the pair $A_{C_0},$ $A_{C_\varkappa}.$ The fact that the operator $A_0$ here is replaced by the operator $A_{C_0}$ reflects the standard argument that the complete scattering theory for a pair of operators requires that the operators forming this pair are ``close enough" to each other.

\end{remark}

The scattering operator $\Sigma$ for the pair $A_\varkappa,$
$A_0$
is defined by
\begin{equation*}
\Sigma=W_{+}^{-1}(A_\varkappa,A_0) W_{-}(A_\varkappa,A_0).
\end{equation*}
The formulae (\ref{eq:formula-0-kappaw-})--(\ref{eq:formula-kappa-0w+}) lead (see ({\it cf.} \cite{MR573902}))
 to the following formula for the action of $\Sigma$ in the model representation:
\begin{equation}
\Phi\Sigma\Phi^*P_K \binom{\widetilde g}{g}=
P_K
\left(\begin{matrix}-(I+\chi_\varkappa^-(S-I))^{-1}(I+\chi_\varkappa^+(S^*-I))g\\[0.7em]
(I+S^*)^{-1}(I+S)(I+\chi_\varkappa^-(S-I))^{-1}(I+\chi_\varkappa^+(S^*-I))g\end{matrix}
\right),
\label{last_formula}
\end{equation}
whenever $\Phi^*P_K \binom{\widetilde g}{g}\in\widetilde N_{\rm e}^0$. This representation holds on a dense linear set in
$\widetilde N_{\rm e}^0$ within the conditions of Theorem \ref{thm:existence-completeness-wave-operators}, which guarantees that
all the objects on the right-hand side of the formula (\ref{last_formula}) are correctly defined.

\section{Spectral representation for the absolutely continuous part of $A_0$}
\label{sec:spectral-repr-ac}

The identity
$$
\biggl\|P_K\binom{\widetilde g}{g}\biggr\|^2_{\mathfrak H}=\bigl\langle(I-S^*S)\widetilde g, \widetilde g \bigr\rangle\qquad\forall P_K \binom{\widetilde g}{g}\in\widetilde N_{\rm e}^0,
$$
which is derived in the same way as in \cite[Section 7]{MR573902} for all
$P_K \binom{\widetilde g}{g}\in\widetilde N_{\rm e}^0,$ which is equivalent to the
condition $(\widetilde g+S^*g)+(S\widetilde g+g)=0,$ see (\ref{lyubimaya_formula}),
allows us to consider the
isometry
$F: \Phi\widetilde N_{\rm e}^0\mapsto L^2({\mathcal K}; I-S^*S)$
defined by
\begin{equation}
FP_K \binom{\widetilde g}{g}=\widetilde g.
\label{F_def}
\end{equation}
Here $L^2({\mathcal K}; I-S^*S)$ is the Hilbert space of ${\mathcal K}$-valued functions on
$\mathbb R$ square summable with the matrix ``weight'' $I-S^*S,$ {\it cf.} (\ref{mathfrakH}). Similarly, the formula
$$
F_*P_K \binom{\widetilde g}{g}= g
$$
defines an isometry $F_*$ from $\Phi\widetilde N_{\rm e}^0$ to $L^2({\mathcal K}; I-SS^*).$

\begin{proposition}[\cite{ChKS_OTAA}]
Suppose that the assumptions of Theorem \ref{thm:existence-completeness-wave-operators} hold. Then the ranges of the operators $F$ and $F_*$ are dense in the spaces $L^2({\mathcal K}; I-S^*S)$ and $L^2({\mathcal K}; I-SS^*),$ respectively.
\end{proposition}

The above statement immediately implies the next result, which allows us to obtain the required spectral representation.

\begin{theorem}
The operator $F,$ respectively $F_*,$ admits an extension to the
unitary mapping between $\Phi N_{\rm e}^0$ and
$L^2({\mathcal K}; I-S^*S),$ respectively $L^2({\mathcal K}; I-SS^*).$
\end{theorem}

It follows that the operator $(A_0-z)^{-1}$ 
considered on $\widetilde N_{\rm e}^0$ acts as the multiplication by
$(k-z)^{-1},$ $k\in{\mathbb R},$
both in $L^2({\mathcal K}; I-S^*S)$ and $L^2({\mathcal K}; I-SS^*)$. In particular,
if one considers the absolutely continuous ``part'' of the
operator $A_0$, namely the operator $A_0^{({\rm e})}:=A_0|_{N_{\rm e}^0},$ then  $F\Phi A_0^{({\rm e})}\Phi^*F^*$ and
$F_*\Phi A_0^{({\rm e})}\Phi^*F_*^*
$
are the operators of multiplication by the independent variable in the
spaces $L^2({\mathcal K}; I-S^*S)$ and $L^2({\mathcal K}; I-SS^*)$, respectively.

In order to obtain a spectral representation from the above result, we need to diagonalise the weights in the definitions of the
above $L^2$-spaces. This diagonalisation is straightforward when
$\alpha=\sqrt{2}I.$ (This choice of $\alpha$ satisfies the conditions of Theorem \ref{thm:existence-completeness-wave-operators} {\it e.g.} when the boundary space $\mathcal K$ is finite-dimensional, which is the case we deal with in the application discussed in Sections \ref{sec:quantum-graphs}, \ref{sec:inverse-scattering}. The corresponding diagonalisation in the general setting will be treated elsewhere.)  In this particular case one has ({\it cf.} (\ref{SM_formula}))
\begin{equation}
S=(M-iI)(M+iI)^{-1},
\label{SviaM}
\end{equation}
and
consequently
\begin{align}
&I-S^*S=-2i (M^*-iI)^{-1}(M-M^*)(M+iI)^{-1},
\label{weightform}\\[0.3em]
&I-SS^*=2i (M+iI)^{-1}(M^*-M)(M^*-iI)^{-1}.\nonumber
\end{align}
Introducing the unitary transformations
\begin{align}
\label{unit_trans1}
&G: L^2({\mathcal K}; I-S^*S)\mapsto L^2\bigl({\mathcal K}; -2i(M-M^*)\bigr),\\[0.3em]
\label{unit_trans2}
&G_*: L^2({\mathcal K}; I-SS^*)\mapsto L^2\bigl({\mathcal K}; -2i(M-M^*)\bigr)
\end{align}
 by the
formulae $g\mapsto (M+iI)^{-1}g$ and
$g\mapsto (M^*-iI)^{-1} g$ respectively, one arrives at the fact that
$
GF\Phi A_0^{({\rm e})}\Phi^* F^*G^*$
and
$G_*F_*\Phi A_0^{({\rm e})}\Phi^*F_*^*G_*^*
$
are the operators of multiplication by the independent
variable in the space $L^2({\mathcal K};-2i(M-M^*))$. We show next that this amounts to the spectral representation in particular in the
case of (non-compact) quantum graphs.

\section{Quantum graphs and their scattering matrices}
\label{sec:quantum-graphs}


The result of the previous section 
only pertains to the absolutely
continuous part of the self-adjoint operator $A_0$, unlike {\it e.g.}
the passage to the classical von Neumann direct integral, under which
the whole of the self-adjoint operator gets mapped to the
multiplication operator in a weighted $L^2$-space (see {\it e.g.}
\cite[Chapter 7]{MR1192782}). Nevertheless, it proves useful in scattering theory, since it 
yields an explicit expression for the scattering matrix $\widehat{\Sigma}$ for the pair
$A_\varkappa,$ $A_0,$ which is the image of the scattering operator
$\Sigma$ in the spectral representation of the operator $A_0.$ Namely, we prove the following statement.

\begin{theorem}
The following formula holds:
\begin{equation}
\label{scat2}
\widehat{\Sigma}=GF\Sigma(GF)^*=
(M-\varkappa)^{-1}(M^*-\varkappa)(M^*)^{-1}M,
\end{equation}
where the right-hand side
represents the operator of multiplication by the corresponding function.
\end{theorem}
\begin{proof}

Using the definition (\ref{F_def}) of the isometry $F$ along with the relationship (\ref{lyubimaya_formula}) between $\widetilde{g}$ and $g$ whenever
$P_K\binom{\widetilde{g}}{g}\in\Phi \widetilde{N}_{\rm e}^\varkappa$ with $\varkappa=0,$
we obtain from (\ref{last_formula}):
\begin{equation}
\label{scat1}
F\Sigma F^*=
\bigl(I+\chi_\varkappa^-(S-I)\bigr)^{-1}\bigl(I+\chi_\varkappa^+(S^*-I)\bigr)(I+S^*)^{-1}(I+S),
\end{equation}
where the right-hand side
represents the operator of multiplication by the corresponding function.

Furthermore, substituting the expression (\ref{S_definition}) for $S$ in terms of $M$ implies that $F\Sigma F^*$ is the operator of multiplication by
\[
(M+iI)(M-\varkappa)^{-1}(M^*-\varkappa)(M^*)^{-1}M(M+iI)
\]
in the space $L^2({\mathcal K}; I-S^*S).$ Using (\ref{weightform}), we now obtain the following identity for all $f, g\in L^2({\mathcal K}; I-S^*S):$
\[
\langle F\Sigma F^*f,g\rangle_{L^2({\mathcal K}; I-S^*S)}=\bigl\langle (I-S^*S)(M+iI)(M-\varkappa)^{-1}(M^*-\varkappa)(M^*)^{-1}M(M+iI)f, g\bigr\rangle
\]
\[
=\bigl\langle -2i (M^*-iI)^{-1}(M-M^*)(M+iI)^{-1}(M+iI)(M-\varkappa)^{-1}(M^*-\varkappa)(M^*)^{-1}M(M+iI)f, g\bigr\rangle
\]
\[
=\bigl\langle -2i(M-M^*)(M-\varkappa)^{-1}(M^*-\varkappa)(M^*)^{-1}M(M+iI)f, (M+iI)g\bigr\rangle,
\]
which is equivalent to (\ref{scat2}), in view of the definition of the operator $G.$
\end{proof}


In applications to quantum graphs it may turn out that the
operator weight $-2i(M-M^*)$ (see (\ref{unit_trans1}), (\ref{unit_trans2})) is degenerate:
more precisely, $M(s)-M(s)^*=2i\sqrt{s}P_{\rm e},$ $s\in{\mathbb R},$ where
$P_{\rm e}$ is the orthogonal projection onto the subspace of ${\mathcal K}$ corresponding
to the set of ``external'' vertices of the graph, {\it i.e.} those vertices
to which semi-infinite edges are attached. Next, we describe the notation pertaining to the quantum graph setting.

\begin{remark}
From this point on, for simplicity of presentation we consider the case of a finite non-compact quantum graph, when the deficiency indices are finite. However, our approach allows us to consider the general setting of infinite deficiency indices, which in  the quantum graph setting leads to an infinite graph. In particular, on could consider the case of an infinite compact part of the graph. 
\end{remark}

In what follows, we denote by ${\mathbb G}={\mathbb G}({\mathcal E},\sigma)$ a
finite metric graph, {\it i.e.} a collection of a finite non-empty set
${\mathcal E}$ of compact or semi-infinite intervals
$e_j=[x_{2j-1},x_{2j}]$ (for semi-infinite intervals we set
$x_{2j}=+\infty$), $j=1,2,\ldots, n,$ which we refer to as
\emph{edges}, and of a partition $\sigma$ of the set of endpoints
${\mathcal V}:=\{x_k: 1\le k\le 2n,\ x_k<+\infty\}$ into $N$ equivalence classes $V_m,$
$m=1,2,\ldots,N,$ which we call \emph{vertices}:
${\mathcal V}=\bigcup^N_{m=1} V_m.$ The degree, or valence, ${\rm deg}(V_m)$ of the vertex $V_m$ is defined as the number of elements in $V_m,$ {\it i.e.} ${\rm card}(V_m).$
Further, we partition the set ${\mathcal V}$ into the two non-overlapping
sets of \emph{internal} ${\mathcal V}^{({\rm i})}$ and \emph{external}
${\mathcal V}^{({\rm e})}$ vertices, where a vertex $V$ is classed as internal
if it is incident to no non-compact edge and external otherwise. Similarly, we partition the set of edges
${\mathcal E}={\mathcal E}^{({\rm i})}\cup {\mathcal E}^{({\rm e})}$, into the collection of
compact (${\mathcal E}^{({\rm i})}$) and non-compact (${\mathcal E}^{({\rm e})}$)
edges. We assume for simplicity that the number of non-compact edges
incident to any graph vertex is not greater than one.

For a finite metric graph ${\mathbb G},$ we consider the Hilbert spaces
$L^2({\mathbb G}):=\bigoplus_{j=1}^n L^2(e_j)$ and
$W^{2,2}({\mathbb G}):=\bigoplus_{j=1}^n W^{2,2}(e_j).$ (Notice that these
spaces do not feel the graph connectivity, as each of them is the same
for different graphs with the same number of edges of the same
lengths.) Further, for a function $f\in W^{2,2}({\mathbb G})$, we
define
the normal derivative at each vertex along each of the adjacent edges,
as follows:
\begin{equation}
\partial_n f(x_j):=\left\{ \begin{array}{ll} f'(x_j),&\mbox{ if } x_j \mbox{ is the left endpoint of the edge},\\[0.35em]
-f'(x_j),&\mbox{ if } x_j \mbox{ is the right endpoint of the
edge.}
\end{array}\right.
\label{co_der}
\end{equation}
In the case of semi-infinite edges we only apply this
definition at the left endpoint of the edge.

\begin{definition}
For $f\in W^{2,2}({\mathbb G})$
and ${a_m}\in{\mathbb C}$ (below referred to as the ``coupling
constant"), the condition of continuity of the function $f$ through
the vertex $V_m$ ({\it i.e.} $f(x_j)=f(x_k)$ if $x_j,x_k\in V_m$)
together with the condition
\[
\sum_{x_j \in V_m} \partial _n f(x_j)={a_m} f(V_m)
\]
is called the $\delta$-type matching at the vertex $V_m$.
\end{definition}

\begin{remark}
  Note that the $\delta$-type matching condition in a particular case
  when ${a_m}=0$ reduces to the standard Kirchhoff matching
  condition at the vertex $V_m$, see {\it e.g.} \cite{MR3013208}.
\end{remark}

\begin{definition}
\label{def6}
The quantum graph Laplacian $A_a,$ $a:=(a_1,...,a_N),$ on a graph ${\mathbb G}$ with
$\delta$-type
 matching conditions is the operator of minus second derivative $-d^2/dx^2$
 in the Hilbert space $L^2({\mathbb G})$ on the domain of functions that
 belong to the Sobolev space $W^{2,2}({\mathbb G})$ and satisfy the
 $\delta$-type matching conditions at every vertex $V_m$,
 $m=1,2,\dots,N.$ The Schr\"odinger operator on the same graph is
 defined likewise on the same domain in the case of summable edge
 potentials ({\it cf.} \cite{MR3484377}).
\end{definition}

If all coupling constants ${a_m}$, $m=1,\dots, N$, are real, it is
shown that the operator $A_a$ is a proper self-adjoint
extension (see (\ref{eq:extension-by-operator})) of a closed
symmetric operator $A$ in $L^2({\mathbb G})$
\cite{MR1459512, MR1671833}. Note that, without loss of generality,
each edge $e_j$ of the graph ${\mathbb G}$ can be considered to be an
interval $[0,l_j]$, where $l_j:=x_{2j}-x_{2j-1}$, $j=1,\dots, n$ is
the length of the corresponding edge. Throughout the present paper we
will therefore only consider this situation.

In \cite{MR3484377}, the following result is obtained for the case of
finite \emph{compact} metric graphs.

\begin{proposition}[\cite{MR3484377}]\label{Prop_M}
  Let ${\mathbb G}$ be a finite compact metric graph with $\delta$-type
  coupling at all vertices. There exists a closed densely defined
  symmetric operator $A$ and a boundary triple such that the
  operator $A_a$ is an almost solvable extension of $A$,
  for which the parametrising matrix $\varkappa$ (see (\ref{eq:extension-by-operator}))
  is given by
  $\varkappa=\mathrm{diag}\{a_1,\dots,a_N\}$, whereas the Weyl
  function is an $N\times N$ matrix with elements
\begin{equation}\label{Eq_Weyl_Func_Delta}
m_{jk}(z)=
\begin{cases}\scriptsize
  -\sqrt{z}\biggl(\sum\limits_{e_p\in E_k}\cot\sqrt{z} l_p- 2\sum\limits_{e_p\in L_k}\tan\dfrac{\sqrt{z}
    l_p}{2}\biggr),
  & j=k,\\
  \sqrt{z}\sum\limits_{e_p\in C_{jk}}\dfrac{1}{\sin\sqrt{z} l_p},
  & j\not=k;\  V_j, V_k\ \mbox{adjacent},\\
  0, & j\not=k;\ V_j, V_k\ \mbox{non-adjacent}.\\
     \end{cases}
\end{equation}
Here the branch of the square root is chosen so that $\Im\sqrt{z}\geq 0,$
$l_p$ is the length of the edge $e_p$,  $E_k$ is the set of non-loop graph edges
incident to the vertex $V_k$, $L_k$ is the set of loops at the
vertex $V_k,$ and $C_{jk}$ is the set of graph edges
connecting vertices $V_j$ and $V_k.$
\end{proposition}

It is easily seen that the rationale of \cite{MR3484377} is applicable
to the situation of non-compact metric graphs.
Indeed, denote by ${\mathbb G}^{({\rm i})}$ the compact part of the graph
${\mathbb G}$, {\it i.e.} the graph ${\mathbb G}$ with all the non-compact edges
removed. Proposition \ref{Prop_M} yields an expression for the Weyl
function $M^{({\rm i})}$ pertaining to the graph ${\mathbb G}^{({\rm i})}$. A simple calculation
then implies the following representation for the $M$-matrix pertaining to the
original graph ${\mathbb G}.$
\begin{lemma}
The matrix functions $M,$ $M^{({\rm i})}$ described above are related by the formula
\begin{equation}
M(z)=M^{({\rm i})}(z) + i\sqrt{z} P_{\rm e},\quad\quad z\in{\mathbb C}_+,
\label{M_Mi}
\end{equation}
where $P_{\rm e}$ is the orthogonal projection in the boundary space
$\mathcal K$ onto the set of external vertices $V_{\mathbb G}^{({\rm e})}$, {\it
  i.e.} the matrix $P_{\rm e}$ such that $(P_{\rm e})_{ij}=1$ if $i=j,$ $V_i\in V_{\mathbb G}^{({\rm e})},$ and
$(P_{\rm e})_{ij}=0$ otherwise.
\end{lemma}
\begin{proof}
Note first that Weyl function of the graph ${\mathbb G}$ for the triple described in Proposition \ref{Prop_M} coincides with the sum of the matrices $M_j(z),$ $j=1,2,\dots, n,$ that are obtained by the formulae
\[
\Gamma_1f=M_j(z)\Gamma_0f,\quad\quad f\in{\rm ker}(A^*-zI),\quad\quad f\equiv 0\ {\rm on}\ {\mathbb G}\setminus e_j.
\]
In order words, the matrix functions $M_j$ describe the Dirichlet-to-Neumann mappings for the data supported on each individual edge $e_j,$ $j=1,2,\dots,n,$ where $A$ is as in Proposition \ref{Prop_M}.

Furthermore, functions $f\in{\rm ker}(A^*-zI)$ that vanish on all edges of the graph ${\mathbb G}$ but one non-compact edge $e_\infty,$ satisfy
\begin{equation}
-f''(x)=zf(x),\quad x\in[0,+\infty), \quad\quad f\in W^{2,2}(0,+\infty),
\label{non_compact_problem}
\end{equation}
where we identify $e_\infty$ and the semi-infinite line $[0,+\infty),$ as well as $f$ and its restriction to $e_\infty.$ Next, all non-trivial solutions to (\ref{non_compact_problem}) have the form
\[
f(x)=f(0)\exp(i\sqrt{z}x),\ \ \ \ x\in[0,+\infty),\ \ \ \ \ \ \ f(0)\neq 0,
\]
for which the value of the co-derivative (\ref{co_der}) at $x=0$ is clearly given by $\partial_nf(0)=i\sqrt{z}f(0).$ Therefore, the corresponding (additive) contribution to
the $M$-matrix, see Definition \ref{def:weyl-function}, is given by the matrix all of whose elements except the diagonal element corresponding to the vertex from which $e_\infty$ emanates are zero, while the only non-zero element equals $(f(0))^{-1}\partial_nf(0)=i\sqrt{z}.$
Repeating this argument for all non-compact edges of ${\mathbb G}$ and using the additivity property for the $M$-matrix discussed above yields the claim.
\end{proof}

The formula  (\ref{M_Mi}) leads to
$M(s)-M^*(s)=2i\sqrt{s}P_{\rm e}$ a.e. $s\in{\mathbb R}$, and the
expression (\ref{scat2}) for $\widehat{\Sigma}$ leads to the classical
scattering matrix $\widehat{\Sigma}_{\rm e}(k)$ of the pair of operators $A_0$ (which is the Laplacian on the graph ${\mathbb G}$ with standard Kirchhoff matching at all the vertices) and $A_\varkappa,$ where $\varkappa=\varkappa=\mathrm{diag}\{a_1,\dots,a_N\}:$
\begin{equation}
\widehat{\Sigma}_{\rm e}(s)=
P_{\rm e} (M(s)-\varkappa)^{-1}(M(s)^*-
\varkappa)(M(s)^*)^{-1}M(s) P_{\rm e},\ \ \ \ s\in{\mathbb R},
\label{sigma_hat}
\end{equation}
which acts as the operator of multiplication in the space
$L^2(P_{\rm e}{\mathcal K}; 4\sqrt{s}ds)$.

\begin{remark}
In the more common approach to the construction of scattering matrices, based on comparing the asymptotic expansions of solutions to spectral equations, see {\it e.g.} \cite{{Faddeev_additional}}, one obtains 
$\widehat{\Sigma}_{\rm e}$ as the scattering matrix. Our approach yields an explicit factorisation of 
$\widehat{\Sigma}_{\rm e}$
into expressions involving the matrices $M$ and $\varkappa$ only, sandwiched between two projections. (Recall that  $M$ and $\varkappa$ contain the information about the geometry of the graph and the coupling constants, respectively.) From the same formula (\ref{sigma_hat}), it is obvious that without the factorisation the pieces of information pertaining to the geometry of the graph and the coupling constants at the vertices are present in the final answer in an entangled form. 

\end{remark}

\begin{remark}
The concrete choice of boundary triple in accordance with Proposition \ref{Prop_M} leads to the fact that the ``unperturbed'' operator $A_0$ is fixed as the Laplacian on the graph with Kirchhoff matching conditions at the vertices. On the other hand, in applications it may be more convenient to consider a formulation where the operator $A_0$ corresponds to some other matching conditions, which would motivate another choice of the triple. This is readily facilitated by the analysis carried out in the preceding sections, {\it cf.} Part 3 of Remark 
\ref{long_remark}.  In particular, we point out that the formula (\ref{scat1}) is written in a triple-independent way.



We reiterate that the analysis above pertains not only to the cases when the coupling constants are real, leading to self-adjoint operators $A_a,$ but also to the case of non-selfadjoint extensions, {\it cf.} Theorem \ref{thm:existence-completeness-wave-operators}.
 \end{remark}

In what follows we often drop the argument $s\in{\mathbb R}$ of the
Weyl function $M$ and the scattering matrices $\widehat{\Sigma},$
$\widehat{\Sigma}_{\rm e}.$ Since
\begin{equation}
(M-\varkappa)^{-1}(M^*-\varkappa)=
I+(M-\varkappa)^{-1}(M^*-M)=I-2i\sqrt{s}(M-\varkappa)^{-1}P_{\rm e}
\label{triple_star}
\end{equation}
and
$$
(M^*)^{-1}M=I+2i\sqrt{s}(M^*)^{-1}P_{\rm e},
$$
a factorisation of $\widehat{\Sigma}_{\rm e}$ into a product of
$\varkappa-$dependent and $\varkappa-$independent factors ({\it cf.}
(\ref{scat2})) still holds in this case in $P_{\rm e} {\mathcal K},$ namely
\begin{equation}
\widehat{\Sigma}_{\rm e}=
\bigl[P_{\rm e} (M-\varkappa)^{-1}(M^*-\varkappa)P_{\rm e}\bigr]\bigl[P_{\rm e}(M^*)^{-1}M P_{\rm e}\bigr].
\label{bigstar}
\end{equation}

\section{Inverse scattering problem for graphs with $\delta$-coupling}
\label{sec:inverse-scattering}
We will now exploit the above approach in the analysis of the inverse
scattering problem for Laplace operators on finite metric graphs,
whereby the scattering matrix $\widehat{\Sigma}_{\rm e}(s),$ defined by (\ref{bigstar}), is assumed
to be known for almost all positive ``energies'' $s\in{\mathbb R},$
along with the graph ${\mathbb G}$ itself. The data to be determined is the
set of coupling constants $\{{a_j}\}_{j=1}^N$. For simplicity, in what follows we treat the inverse problem for graphs with
real coupling constants, which corresponds to self-adjoint operators, leaving the non-selfadjont situation to be addressed elsewhere.

First, given $\widehat{\Sigma}_{\rm e}(s)$ for almost all $s>0$, we reconstruct the meromorphic matrix-function
$P_{\rm e}(M^{({\rm i})}(z)-\varkappa)^{-1}P_{\rm e}$ for all complex $z,$
excluding the poles. This is an explicit calculation based on the
second resolvent identity (see {\it e.g.} \cite[Thm.\,5.13]{MR566954}). Namely, almost everywhere on the positive
half-line one has
\begin{multline*}
(M-\varkappa)^{-1}=
(M^{({\rm i})}-\varkappa)^{-1}-(M-\varkappa)^{-1}(M-M^{({\rm i})})(M^{({\rm i})}-\varkappa)^{-1}
\\[0.4em]
=\bigl[I-(M-\varkappa)^{-1}(M-M^{({\rm i})})\bigr](M^{({\rm i})}-\varkappa)^{-1},
\end{multline*}
and hence
\begin{equation}
P_{\rm e}(M-\varkappa)^{-1}P_{\rm e} =
\bigl[P_{\rm e}-i\sqrt{s}P_{\rm e} (M-\varkappa)^{-1}P_{\rm e}\bigr]
P_{\rm e}(M^{({\rm i})}-\varkappa)^{-1}P_{\rm e}.
\label{triangle}
\end{equation}

  Further, the first factor on
the right-hand side of (\ref{triangle}) is invertible for almost all $s>0.$
Indeed, we note first that
$\widehat{\Sigma}_{\rm e}^\varkappa:=P_{\rm e}
(M(s)-\varkappa)^{-1}(M(s)^*-\varkappa)$ is unitary in
$P_{\rm e}\,{\mathcal K}$ for almost all $s>0,$ since
\begin{multline*}
(M-\varkappa)(M^*-\varkappa)^{-1} (M-M^*) (M-\varkappa)^{-1}(M^*-\varkappa)\\[0.3em]
=(M-\varkappa)(M^*-\varkappa)^{-1}\bigl[(M-\varkappa)-(M^*-\varkappa)\bigr](M-\varkappa)^{-1}(M^*-\varkappa)\\[0.3em]
=(M-\varkappa)-(M^*-\varkappa)=M-M^*
\end{multline*}
and $M-M^*=2i\sqrt{s}P_{\rm e}$. Now, since
$$
P_{\rm e}-i\sqrt{s} P_{\rm e} (M-\varkappa)^{-1}P_{\rm e}=\bigl(I+\widehat{\Sigma}_{\rm e}^\varkappa\bigr)/2
$$
it suffices to show that $-1$ is not an eigenvalue of
$\widehat{\Sigma}_{\rm e}^\varkappa(s)$ for almost all $s>0.$ Assume the
opposite, {\it i.e.} for some $s>0$ one has
$$
(M(s)^*-\varkappa)^{-1}u_s=-(M(s)-\varkappa)^{-1}u_s,\ \ \ \ \ \ u_s\in P_{\rm e}\,\mathcal K\setminus\{0\}.
$$
A straightforward calculation then yields
$$
(M(s)^*-\varkappa)^{-1}(M^{({\rm i})}(s)-\varkappa)(M(s)-\varkappa)^{-1}u_s=0,
$$
from where
$$
(M(s)-\varkappa)^{-1}u_s\in \mathrm{ker}\bigl(M^{({\rm i})}(s)-\varkappa\bigr).
$$
The latter kernel is non-trivial only at the points $s$ which belong
to the (discrete) spectrum of the Laplacian on the compact part
${\mathbb G}^{({\rm i})}$ of the graph ${\mathbb G}$. It follows that
$(M(s)-\varkappa)^{-1}u_s$ is zero for almost all $s>0,$ which is a
contradiction with $u_s\neq 0.$

Note that, for a given graph ${\mathbb G},$ the expression
$P_{\rm e}(M-\varkappa)^{-1}P_{\rm e}$ is found by combining
(\ref{triple_star}) and (\ref{bigstar}):
\begin{equation}
P_{\rm e}(M-\varkappa)^{-1}P_{\rm e}=\frac{1}{2i\sqrt{s}}\bigl(P_{\rm e}-\widehat{\Sigma}_{\rm e}[P_{\rm e}(M^*)^{-1}MP_{\rm e}]^{-1}\bigr),
\label{interm_form}
\end{equation}
where we treat both $[P_{\rm e}(M^*)^{-1}MP_{\rm e}]^{-1}$ and, as before, $\widehat{\Sigma}_{\rm e}$ as operators in $P_{\rm e}{\mathcal K}.$

It follows from (\ref{triangle}) and (\ref{interm_form}) that for given $M,$ $\widehat{\Sigma}_{\rm e}$ the expression $P_{\rm e}(M^{({\rm i})}-\varkappa)^{-1}P_{\rm e}$ is determined uniquely for almost all $s>0:$
\[
P_{\rm e}(M^{({\rm i})}-\varkappa)^{-1}P_{\rm e}
=\bigl[P_{\rm e}-i\sqrt{s}P_{\rm e} (M-\varkappa)^{-1}P_{\rm e}\bigr]^{-1} P_{\rm e}(M-\varkappa)^{-1}P_{\rm e}
\]
\[
=\frac{1}{i\sqrt{s}}\bigl(P_{\rm e}+\widehat{\Sigma}_{\rm e}[P_{\rm e}(M^*)^{-1}MP_{\rm e}]^{-1}\bigr)^{-1}\bigl(P_{\rm e}-\widehat{\Sigma}_{\rm e}[P_{\rm e}(M^*)^{-1}MP_{\rm e}]^{-1}\bigr)
\]
\begin{equation}
=\frac{1}{i\sqrt{s}}\biggl(2\bigl(P_{\rm e}+\widehat{\Sigma}_{\rm e}[P_{\rm e}(M^*)^{-1}MP_{\rm e}]^{-1}\bigr)^{-1}-I\biggr)P_{\rm e}.
\label{DtD}
\end{equation}
In particular, due to the property of analytic continuation, the expression $P_{\rm e}(M^{({\rm i})}-\varkappa)^{-1}P_{\rm e}$ is determined uniquely in the whole of $\mathbb C$ with the
exception of a countable set of poles, which coincides with the
set of eigenvalues of the self-adjoint Laplacian $A_\varkappa^{({\rm i})}$ on the
compact part ${\mathbb G}^{({\rm i})}$ of the graph ${\mathbb G}$ with matching
conditions at the graph vertices given by the matrix $\varkappa,$ {\it cf.} Proposition \ref{Prop_M}.

\begin{definition}
\label{definition7}
Given a partition ${\mathcal V}_1\cup{\mathcal V}_2$ of the set of graph vertices, for $z\in{\mathbb C}$ consider the linear set $U(z)$ of functions that satisfy the differential equation $-u_z''=zu_z$ on each edge, subject to the conditions of continuity at all vertices of the graph and the $\delta$-type matching conditions at the vertices in the set ${\mathcal V}_2.$ For each function $f\in U(z),$ consider the vectors
\[
\Gamma_1^{{\mathcal V}_1}u_z:=\Bigl\{\sum_{x_j \in V_m} \partial _n f(x_j)\Bigr\}_{V_m\in{\mathcal V}_1},\quad\quad\Gamma_0^{{\mathcal V}_1}u_z:=\bigl\{f(V_m)\bigr\}_{V_m\in{\mathcal V}_1}.
\]
The {\it Robin-to-Dirichlet map} of the set ${\mathcal V}_1$ maps the vector  $(\Gamma_1^{{\mathcal V}_1}-\varkappa^{{\mathcal V}_1}\Gamma_0^{{\mathcal V}_1})u_z$ to $\Gamma_0^{{\mathcal V}_1}u_z,$ where $\varkappa^{{\mathcal V}_1}:=\diag \{a_m:\ V_m\in{\mathcal V}_1\}$. (Note that the function $u_z\in U(z)$ is determined uniquely by $(\Gamma_1^{{\mathcal V}_1}-\varkappa^{{\mathcal V}_1}\Gamma_0^{{\mathcal V}_1})u_z$ for all $z\in\mathbb C$ except a countable set of real points accumulating to infinity).
\end{definition}

\begin{remark} The above definition is a natural generalisation of the corresponding definitions of Dirichlet-to-Neumann and Neumann-to-Dirichlet maps pertaining to the graph boundary, considered
  in {\it e.g.} \cite{MR3013208}, \cite{MR2600145}.
\end{remark}

We argue that the matrix $P_{\rm e}(M^{({\rm i})}-\varkappa)^{-1}P_{\rm e}$ is the Robin-to-Dirichlet map
for the set ${\mathcal V}^{({\rm e})}.$ Indeed, assuming $\phi:=\Gamma_1 u_z-\varkappa \Gamma_0 u_z$ and $\phi=P_{\rm e} \phi,$ where the latter condition ensures the correct $\delta$-type matching on the set $\mathcal{V}^{({\rm i})},$ one has
$P_{\rm e}\phi=(M^{({\rm i})}-\varkappa)\Gamma_0 u_z$ and hence $\Gamma_0 u_z=(M^{({\rm i})}-\varkappa)^{-1}P_{\rm e} \phi$. Applying $P_{\rm e}$ to the last identity yields the claim, in accordance with Definition \ref{definition7}.


We have thus proved the following theorem.
\begin{theorem}
\label{thm9.1}

  The Robin-to-Dirichlet map for the vertices ${\mathcal V}^{({\rm e})}$
  is determined uniquely by the scattering matrix
  $\widehat{\Sigma}_{\rm e}(s),$ $s\in{\mathbb R},$ via the formula (\ref{DtD}).

\end{theorem}


The following definition, required for the formulation of the next theorem, is a generalisation of the procedure of graph contraction, well studied in the algebraic graph theory, see {\it e.g.} \cite{Tutte}.

\begin{definition}[Contraction procedure\footnote{One of the referees pointed out that this procedure is sometimes referred to a ``layer peeling". We have opted to keep the term ``contraction'' for it, in line with the terminology of the algebraic literature.} for graphs and associated quantum graph Laplacians]
\label{def8}
For a given graph ${\mathbb G}$ vertices $V$ and $W$ connected  by an edge $e$ are ``glued'' together to form a
new vertex $(VW)$
of the contracted graph $\widetilde{\mathbb G}$ while simultaneously the edge $e$ is removed,  whereas the rest of the graph remains
unchanged. We do allow the situation of multiple edges, when $V$ and
$W$ are connected in ${\mathbb G}$ by more than one edge, in which case all such edges but the edge $e$ become loops of their
respective lengths attached to the vertex $(VW)$. The corresponding quantum graph Laplacian $A_a$ defined on ${\mathbb G}$ is contracted to the
quantum graph Laplacian $\widetilde{A}_{\widetilde{a}}$ by the application of the following rule pertaining to the coupling constants:
a coupling constant
at any unaffected vertex remains the same, whereas the coupling
constant at the new vertex $(VW)$ is set to be the sum of the coupling constants at $V$ and $W.$
Here it is always assumed that all quantum graph Laplacians are described by Definition \ref{def6}.
\end{definition}

The matrix $\varkappa$ of the coupling constants is now determined as part of an iterative procedure based on the following
result.

\begin{theorem}
\label{NtD_NtD}
Suppose that the edge lengths of the graph ${\mathbb G}^{({\rm i})}$ are rationally independent.
The element\footnote{By renumbering if necessary, this does not lead to loss of generality.} $(1,1)$  of the Robin-to-Dirichlet map described above yields the element $(1,1)$ of the ``contracted'' graph
$\widetilde{\mathbb G}^{({\rm i})}$ obtained from the graph ${\mathbb G}^{({\rm i})}$ by removing a non-loop edge $e$
emanating from $V_1.$ The procedure of passing from the graph ${\mathbb G}^{({\rm i})}$ to the contracted graph
$\widetilde{\mathbb G}^{({\rm i})}$ is given in Definition \ref{def8}.

\end{theorem}
\begin{proof}

Due to the assumption
that the edge lengths of the graph ${\mathbb G}^{({\rm i})}$ are rationally
independent, the element (1,1), which we denote by $f_1,$ is expressed explicitly as a function of
$\sqrt{z}$ \emph{and} all the edge lengths $l_j,$ $j=1,2,\dots, n,$ in particular, of the length of the edge $e,$ which we assume  to be $l_1$ without loss of generality. This is
an immediate consequence of the explicit form of the matrix
$M^{({\rm i})},$ see (\ref{Eq_Weyl_Func_Delta}). Again without loss of generality, we also assume that the edge $e$ connects the vertices $V_1$ and $V_2.$

Further, consider the expression $\lim_{l_1\to 0}
f_1(\sqrt{z}; l_1,\dots,l_n; a)$. On the one hand, this limit is known from the explicit expression for $f_1$ mentioned above. On the
other hand, $f_1$ is the ratio of the determinant ${\mathcal D}^{(1)}(\sqrt{z}; l_1, \dots, l_n; a)$ of the principal minor of the matrix $M^{({\rm i})}(z)-\varkappa$ obtained by removing its first row and and first column and the
determinant of $M^{({\rm i})}(z)-\varkappa$ itself:
\[
f_1(\sqrt{z}; l_1,\dots,l_n; a)=\frac{{\mathcal D}^{(1)}(\sqrt{z}; l_1, \dots, l_n; a)}{{\rm det}\bigl(M^{({\rm i})}(z)-\varkappa\bigr)}
\]
Next, we multiply by $-l_1$ both the numerator and denominator of this ratio, and pass to the limit in each of them separately:
\begin{equation}
\lim_{l_1\to0}f_1(\sqrt{z}; l_1,\dots,l_n; a)=\frac{\lim\limits_{l_1\to0}(-l_1){\mathcal D}^{(1)}(\sqrt{z}; l_1, \dots, l_n; a)}{\lim\limits_{l_1\to0}(-l_1){\rm det}\bigl(M^{({\rm i})}(z)-\varkappa\bigr)}
\label{bulky_ratio}
\end{equation}
The numerator of (\ref{bulky_ratio}) is easily computed as the
determinant ${\mathcal D}^{(2)}(z; l_1, \dots, l_n; a)$ of the minor of $M^{({\rm i})}(z)-\varkappa$ obtained by removing its first two rows and first two columns.

As for the denominator of (\ref{bulky_ratio}), we add to the second row of the matrix
$M^{({\rm i})}(z)-\varkappa$ its first row multiplied by $\cos(\sqrt{z}l_1),$ which leaves the determinant unchanged.  This operation, due to the identity
\[
-\cot(\sqrt{z}l_1)\cos(\sqrt{z}l_1)+\frac{1}{\sin(\sqrt{z}l_1)}=\sin(\sqrt{z}l_1),
\]
cancels out the singularity of all matrix elements of the second row at the point $l_1=0.$ We introduce the factor $-l_1$ ({\it cf.} \ref{bulky_ratio}) into the first row and pass to the limit as $l_1\to 0.$ Clearly, all rows but the first are regular at $l_1=0$ and hence converge to their limits as $l_1\to0.$ Finally, we add to the second column of the limit its first column, which again does not affect the determinant, and note that the first row of the resulting matrix has
one non-zero element, namely the $(1,1)$ entry. This procedure reduces the denominator in (\ref{bulky_ratio}) to the determinant of a matrix of the size reduced by one.  As in \cite{MR3430381}, it is checked that this determinant is nothing but ${\rm det}(\widetilde{M}^{({\rm i})}-\widetilde\varkappa)$, where $\widetilde M^{({\rm i})}$ and $\widetilde\varkappa$ are the Weyl matrix and the (diagonal) matrix of coupling constants pertaining to the contracted graph $\widetilde{\mathbb G}^{({\rm i})}.$ This immediately implies that the ratio obtained as a result of the above procedure coincides with the entry (1,1) of the matrix $(\widetilde M^{({\rm i})}-\widetilde\varkappa)^{-1},$ {\it i.e.}
\begin{equation}
\lim_{l_1\to0}f_1(\sqrt{z}; l_1,\dots,l_n; a)=f_1^{(1)}(\sqrt{z}; l_2,\dots,l_n; \widetilde{a}),
\label{imya}
\end{equation}
where $f_1^{(1)}$ is the element (1,1) of the Robin-to-Dirichlet map of the contracted graph $\widetilde{\mathbb G}^{({\rm i})},$ and $\widetilde{a}$ is given by Definition \ref{def8}.
\end{proof}

The main result of this section is the theorem below, which is a corollary of Theorems \ref{thm9.1} and \ref{NtD_NtD}. We assume without loss of generality that $V_1\in{\mathcal V}^{{\rm (e)}}$ and denote by $f_1(\sqrt{z})$ the (1,1)-entry of the Robin-to-Dirichlet map for the set ${\mathcal V}^{{\rm (e)}}.$ We set the following notation. Fix a spanning tree ${\mathbb T}$ (see {\it e.g.} \cite{Tutte}) of the graph ${\mathbb G}^{({\rm i})}.$ We let the vertex $V_1$ to be the root of $\mathbb T$ and assume, again without loss of generality, that the number of edges in the path $\gamma_m$ connecting $V_m$ and the root is a non-decreasing function of $m.$
Denote by $N^{(m)}$ the number of vertices in the path $\gamma_m,$
and by $\bigl\{l^{(m)}_k\bigr\},$ $k=1,\dots, N^{(m)}-1,$ the associated sequence of lengths of the edges in $\gamma_m,$ ordered along the path from the root $V_1$ to
$V_m.$ Note that each of the lengths $l^{(m)}_k$ is clearly one of the edge lengths $l_j$ of the compact part of the original graph ${\mathbb G}.$

\begin{theorem}
\label{last_theorem}
Assume that the graph ${\mathbb G}$ is connected and the lengths of its compact edges are rationally independent. Given the scattering matrix $\widehat{\Sigma}_{\rm e}(s),$ $s\in{\mathbb R},$ the Robin-to-Dirichlet map for the set ${\mathcal V}^{{\rm (e)}}$ and the matrix of coupling constants $\varkappa$ are determined constructively in a unique way. Namely,
the following formulae hold for $l=1, 2,\dots, N$ and determine $a_m,$ $m=1,\dots, N:$
\[
\sum_{m: V_m\in\gamma_l}a_m
=\lim_{\tau\to+\infty}\Biggl\{-\tau\Bigl(\sum_{V_m\in\gamma_l}{\rm deg} (V_m)-2(N^{(l)}-1)\Bigr)-\frac{1}{f_1^{(l)}(i\tau)}\Biggr\},
\]
where
\begin{equation}
f_1^{(l)}(\sqrt{z}):=\lim_{l^{(l)}_{N^{(l)}-1}\to0}\dots\lim_{l^{(l)}_2\to0}\lim_{l^{(l)}_1\to0}f_1(\sqrt{z}),
\label{last}
\end{equation}
where in the case $l=1$ no limits are taken in (\ref{last}).
\end{theorem}

\begin{proof}


We first apply Theorem \ref{thm9.1} to determine the Robin-to-Dirichlet map for the vertices ${\mathcal V}^{({\rm e})}.$ Next,
 we notice that the knowledge of the (1,1)-element $f_1$ of the Robin-to-Dirichlet map for the set ${\mathcal V}^{({\rm e})},$ {\it i.e.} of the matrix $P_{\rm e}(M^{({\rm i})}-\varkappa)^{-1}P_{\rm e},$ together with the asymptotic expansion for
  $M^{({\rm i})}(z)$ as $\sqrt{z}\to + i\infty,$ yields the element (1,1) of
  the matrix $\varkappa$, which is the coupling constant ${a_1}$ at
  the vertex $V_1,$ see Proposition \ref{Prop_M}. Indeed, setting
  $\sqrt{z}=i\tau,$ $\tau\to +\infty,$ one has ({\it cf.} (\ref{Eq_Weyl_Func_Delta}))
\begin{equation}
\frac 1 {f_1}= i\tau \biggl(-\sum_{e_p\in E_1} \cot(i\tau l_p)+2
\sum_{e_p\in L_1}\tan\frac{i\tau l_p}{2}\biggr)-{a_1} + o\bigl(\tau^{-K}\bigr)
\label{alpha1_recovery_0}
\end{equation}
\begin{equation}
=-\tau\,{\rm deg}(V_1)
-{a_1} + o\bigl(\tau^{-K}\bigr)
, \ \ \ \ \tau\to+\infty
\label{alpha1_recovery}
\end{equation}
for all $K>0$,
where the first sum in (\ref{alpha1_recovery_0}) is taken over all non-loop edges
$e_p$ of ${\mathbb G}^{({\rm i})}$ emanating from the vertex $V_1$ and the second
over all loops $e_p$ attached to $V_1.$
The coupling constant ${a_1}$
is then recovered directly from (\ref{alpha1_recovery}).

In order to determine the coupling constant $a_2,$ we apply
Theorem \ref{NtD_NtD}. In order to do so we note that the
the vertex $V_2$ is connected to $V_1$ by the edge of the length $l_1^{(2)}$
and apply the contraction procedure along this edge. In particular, the formula (\ref{imya}), together with asymptotics (\ref{alpha1_recovery}) re-written for the first diagonal element of the contracted graph, yields the coupling constant pertaining to the vertex $\widetilde{V}_1:=(V_1V_2)$ of the contracted graph, which, by Theorem \ref{NtD_NtD}, is equal to $a_1+a_2:$
\[
a_1+a_2=\lim_{\tau\to+\infty}\Biggl\{i\tau \biggl(-\sum_{e_p\in\widetilde{E}_1} \cot(i\tau l_p)+2
\sum_{e_p\in\widetilde{L}_1}\tan\frac{i\tau l_p}{2}\biggr)-\frac{1}{f_1^{(1)}}\Biggr\}
\]
\begin{equation}
=\lim_{\tau\to+\infty}\Biggl\{-\tau \bigl({\rm deg} (V_1)+{\rm deg}(V_2)-2\bigr)
-\frac{1}{f_1^{(1)}}\Biggr\},
\label{imya2}
\end{equation}
where $\widetilde{E}_1$ is the set of all non-loop edges of the contracted graph $\widetilde{\mathbb G}^{({\rm i})}$ emanating from the vertex $\widetilde{V}_1,$ $\widetilde{L}_1$ is the set of loops attached to this same vertex, and $f_1^{(1)},$ explicitly given by (\ref{imya}), is the element (1,1) of the Robin-to-Dirichlet map of the contracted graph. Thus we recover the value of the coupling constant ${a_2},$ as a result of consequent evaluations of indeterminate forms of two different types: ``$0/0$" (see (\ref{imya})) and
``$\infty-\infty$" (see (\ref{imya2})).

Since the graph ${\mathbb G}$ is connected, the above procedure is iterated until the only remaining vertex of the contracted graph is $V_1,$ at which point the last coupling constant $a_N$ is determined. The claim of the theorem follows.
\end{proof}

\begin{remark}
1. Notice that each step of the above iterative process generates a set of loops, which is treated according to the formula (\ref{alpha1_recovery_0}). Alternatively, these loops can be discarded by an elementary recalculation of the corresponding element of the Robin-to-Dirichlet map in the application of Theorem \ref{NtD_NtD}.

2. From the proof of Theorem \ref{last_theorem} it actually follows that the inverse problem of determining matching conditions based on the Robin-to-Dirichlet map pertaining to any subset of graph vertices for any finite and compact graph ${\mathbb G}$ has a unique and constructive solution. As in the theorem, the graph is assumed connected and its edge lengths rationally independent.  More than that, for the solution of the named inverse problem it suffices to know any one diagonal element of the Robin-to-Dirichlet map.

\end{remark}

\section*{Acknowledgements}


KDC is grateful for the financial support of
the Engineering and Physical Sciences Research Council: Grant EP/L018802/2 ``Mathematical foundations of metamaterials: homogenisation, dissipation and operator theory''. AVK has been partially supported by a grant of the Ukrainian Ministry for Education and by the RFBR grant 16-01-00443-a. LOS has been partially supported by UNAM-DGAPA-PAPIIT IN105414 and SEP-CONACYT 254062.

We also thank the reviewers for a number of useful suggestions, which have helped us improve the manuscript.

\def\cprime{$'$} \def\lfhook#1{\setbox0=\hbox{#1}{\ooalign{\hidewidth
  \lower1.5ex\hbox{'}\hidewidth\crcr\unhbox0}}}

\end{document}